\documentclass[reqno]{amsart}
\usepackage{amssymb}
\usepackage{amsmath}
\usepackage{amsthm}
\usepackage{mathtools}
\usepackage{graphicx}
\usepackage{comment}
\usepackage{caption}
\usepackage{calc}
\usepackage[font=small,labelfont=small]{caption}
\usepackage[style=numeric,maxbibnames=99]{biblatex}
\DeclareNameAlias{default}{last-first}

\usepackage{tikz}
\usepackage{tikz-3dplot}
\tdplotsetmaincoords{70}{110} % <-- this defines tdplot_main_coords

\usetikzlibrary{calc}
\usetikzlibrary{arrows}

\usepackage{hyperref}
\hypersetup{colorlinks}

\usepackage{amssymb, mathrsfs, amsfonts, amsmath}

\makeatletter
\def\part#1{%
  \refstepcounter{part}%
  \bigskip
  \begin{center}
    \normalfont\bfseries
    \Large Part~\thepart: #1
  \end{center}
  \medskip
}
\makeatother

\usepackage{tikz}
\usetikzlibrary{arrows}

\usepackage[english]{babel}

\usepackage[margin=1.4in, centering]{geometry}

\DeclareSymbolFont{bbold}{U}{bbold}{m}{n}
\DeclareSymbolFontAlphabet{\mathbbold}{bbold}
\newtheorem{thm}{Theorem}
\newtheorem{prop}[thm]{Proposition}
\newtheorem{lem}[thm]{Lemma}
\newtheorem{cor}[thm]{Corollary}
\theoremstyle{definition}
\newtheorem{defn}[thm]{Definition}
\newtheorem{ex}[thm]{Example}
\theoremstyle{remark}
\newtheorem{rem}[thm]{Remark}

\newtheoremstyle{plain-bold-number}% name
  {6pt}{6pt}% Space above/below
  {\itshape}% Body font
  {}% Indent
  {\bfseries}% Head font (Theorem)
  {.}% Punctuation after head
  { }% Space after head
  {\thmname{#1}\thmnumber{ \bfseries #2}\thmnote{ (#3)}}% Head spec
\makeatother

\theoremstyle{plain-bold-number}
\newtheorem{atheorem}{Theorem}

\usepackage{graphicx} % Required for inserting images

\newcommand{\R}{\mathbb{R}}

\newcommand{\N}{\mathbb{N}}

\tikzset{vtx/.style={circle, fill, inner sep=1.5pt}}

\title{Falconer-type results for any finite graph with multiple pins}
\author[T. Borges, B. Foster, Y. Ou, E. Palsson, F. Romero Acosta]{Tainara Borges, Ben Foster, Yumeng Ou, Eyvindur Palsson, Francisco Romero Acosta}

\address[T. Borges]{Department of Mathematics, University of Pennsylvania, Philadelphia, PA 19104, USA}\email{tborges@sas.upenn.edu}
\address[B. Foster]{Department of Mathematics, Washington University in St. Louis, St. Louis, MO 63130, USA}\email{bfoster@wustl.edu}
\address[Y. Ou]{Department of Mathematics, University of Pennsylvania, Philadelphia, PA 19104, USA}\email{yumengou@sas.upenn.edu}
\address[E. Palsson]{Department of Mathematics, Virginia Tech, Virginia, VA 24061, USA}\email{palsson@vt.edu}
\address[F. Romero Acosta]{MSI, The Australian National University, ACT, Australia}\email{Francisco.Romero@anu.edu.au}

\begin{document}
\maketitle

\begin{abstract}
A generalization of the celebrated Falconer distance problem asks for a graph $G=(\mathcal{V},\mathcal{E})$, with vertex set $\mathcal{V}$ and edge set $\mathcal{E}$, how large the Hausdorff dimension of a compact set $E\subset \mathbb{R}^d$, $d\geq 2$, needs to be to guarantee that the distance graph
$$ \Delta^{G}(E):= \lbrace (|x_{i}-x_{j}|)_{(v_i,v_j)\in\mathcal{E}} : x_1,\ldots,x_{|\mathcal{V}|}\in E \rbrace $$
has positive $|\mathcal{E}|$-dimensional Lebesgue measure. Here we represent the edges in $\mathcal{E}$ as ordered pairs of vertices $(v_i,v_j)$ with $i<j$. Many results exist for particular graphs, such as trees and simplices. Some general results exist, but they require intricate calculations, such as computing Fourier decay of the natural measure on the configuration set or mapping properties of associated Fourier integral operators. In this paper, using the graph theory notion of $k$-degeneracy, which is easy to compute, we obtain a non-trivial dimensional threshold $\frac{d+k}{2}$, $d>k$, for any non-trivial graph $G$. Key ingredients for our result are identifying pinned stars as the right building blocks for a general graph as well as refining a Fubini type argument due to Taylor and the third named author. We further generalize this to graphs with multiple pins by introducing the $k$-admissibility of a graph, a generalization of $k$-degeneracy that takes pins into account, as well as by extending the Fubini argument to the multiple pinned setting. Not only do we obtain non-trivial results in high enough dimensions for any distance graph, but for particular graphs (such as cycles) our results are also strong and improve the previously best known results. Our methods extend to general two point configurations, contingent on results being available for the appropriate star building blocks. %In the hypergraph setting, where one considers point configurations with multiple points, we showcase how our methods extend to that setting through looking at the particular configuration that calculates areas of triangles. Recently, Shmerkin and Yavicoli resolved in a strong way a conjecture about the dimensional threshold needed for obtaining a positive Lebesgue measure for areas of triangles in the plane. We need refined versions of their results in the plane and obtain results in higher dimensions that are of independent interest. These results form building blocks that allow us to obtain results for many different types of graphs where triangles are connected and areas are computed.
\end{abstract}

\section{Introduction}

The celebrated Falconer distance problem, which can be viewed as a continuous analogue of the Erd\H{o}s distinct distance problem, asks for a compact set $E\subset\mathbb{R}^d$, $d\geq 2$, how large its Hausdorff dimension $\dim_{H}(E)$ needs to be to guarantee that its distance set
$$ \Delta(E) = \lbrace |x-y| : x,y\in E \rbrace $$
has positive Lebesgue measure. Falconer showed the threshold $\frac{d}{2}$ was necessary and conjectured it was sufficient \cite{Falconer85}. In recent years there has been much progress on this conjecture with the current best thresholds being $\frac{5}{4}$ when $d=2$ and $\frac{d}{2} + \frac{1}{4} - \frac{1}{8d+4}$ when $d\geq 3$ \cite{GIOW20,DORZ23}. A classic variant, the pinned Falconer distance problem, similarly asks how large $\dim_{H}(E)$ needs to be to guarantee the existence of a pin $x\in E$ such that the pinned distance set
$$ \Delta_{x}(E) = \lbrace |x-y| : y\in E \rbrace $$
has positive Lebesgue measure. Any threshold established for the pinned problem implies the same threshold works for the original problem. Surprisingly, due to a conversion mechanism by Liu \cite{Liu19}, all the best thresholds for the Falconer distance problem also hold for the pinned variant. Yet another classic variant is asking for the distance set $\Delta(E)$ to have non-empty interior, instead of merely positive Lebesgue measure. Mattila and Sj\"{o}lin \cite{MattilaSjolin} obtained the threshold $\frac{d+1}{2}$, $d\geq 2$, for this problem and no improvements have been made on it since their result. Finally one can even combine variants and ask for non-empty interior of the pinned distance set $\Delta_{x}(E)$ with the best results being due to Peres and Schlag \cite{PS00} with a threshold $\frac{d+2}{2}$, $d\geq 4$, while the first four authors of this paper recently obtained the improved threshold $\frac{12}{5}$ when $d=3$ and the first non-trivial threshold $\frac{7}{4}$ when $d=2$ \cite{pinnedtrees}.

All of the above variants have in common that they focus on the classic Euclidean distance between two points. A big push in the area has been to extend these results to other two point configurations, such as dot products, as well as multipoint configurations, such as triangles and simplices. Such generalizations can help us discover better viewpoints on problems, with an example being the modern proof of the Mattila-Sj\"{o}lin result, as presented, for example, in the classic book by Mattila \cite{Mattilabook2015}, having its origin in a result due to Iosevich, Mourgoglou and Taylor \cite{IMT12} where they extended the result to norms induced by symmetric bounded convex bodies with a smooth boundary and everywhere non-vanishing Gaussian curvature. Generalizations can even lead to breakthroughs in the classical variants, one example being the group action method by Greenleaf, Iosevich, Liu and the fourth named author \cite{GILP15}, developed originally for triangles and simplices but which gives a new viewpoint on the original distance problem and was a key ingredient in the magic formula, due to Liu \cite{Liu19}, that connected the pinned distance problem to the original unpinned distance problem. The literature on such generalizations is vast, so we restrict ourselves to highlighting some main themes. One theme is to consider more general two point configurations $\Phi:\mathbb{R}^{d}\times\mathbb{R}^{d}\rightarrow\mathbb{R}^k$ and their corresponding configuration sets
$$ \Delta^{\Phi}(E) = \lbrace \Phi(x,y) : x,y\in E \} $$
with a good example being the dot product, namely $\Phi(x,y)=x\cdot y$, as studied in \cite{EIT11,EHI13,BMS24}, but see also \cite{IMT12,GIT21}. Another theme is to look at larger configurations, such as triangles and simplices, which was initiated in this setting by \cite{GI12}, although similar questions have been studied for decades in the setting of the Erd\H{o}s distance problem \cite{BMP05}. Just to indicate some configurations that have been studied, then further improvements on triangles and simplices have appeared in \cite{EHI13,GILP15,PRA23,PRA25,IPPS22}, cycles in \cite{GIP17,IMMM25}, areas and volumes appeared in \cite{McDonald21,GM22,SY25}, angles were studied in \cite{Harangi11,HKKMMMS13,IMP16,Mathe17} and somewhat general configurations were handled in \cite{GGIP15,GIT22,GIT24,GIT25} using the language of $\Phi$ configurations.

The configuration function for a triangle $\Phi_{\text{triangle}}(x,y,z) = (|x-y|,|y-z|,|z-x|)$ encodes the side lengths of the triangle and is just built out of the usual distances that appear in the original Falconer distance problem. A natural question arises whether there is a mechanism that takes results on the Falconer distance problem or other natural building blocks and automatically leads to thresholds for triangles or even more generally to any configuration built out of distances. In the setting of chains and trees of distances, such a result has been achieved using distances as building blocks, with the first result being due to Bennett, Iosevich, Taylor \cite{BIT16} where they obtained the threshold $\frac{d+1}{2}$ for getting non-empty interior of chains of distances. This was later extended to trees \cite{IT19}, pinned trees \cite{OT22} and most recently to non-empty interior of pinned trees \cite{pinnedtrees}. All these results are inductive in nature and crucially rely on there being no cycles in the configuration. In \cite{IPPS22} Iosevich, Pham, Pham and Shen used the building block of a $k$-star to obtain results on a pinned $k$-simplex but didn't build any further than that. Finally, a recent result by Greenleaf, Iosevich and Taylor \cite{GIT25} was able to implement an inductive strategy for hypergraphs, where for example they show a result on nonempty interior for a chain of triangles. Their building block is their base shape and then they can get a tree of such configurations through an inductive strategy, but anything with a cycle (other than the basic building block) is out of reach.

In contrast to these inductive strategies the results of Grafakos, Greenleaf, Iosevich and the fourth listed author \cite{GGIP15} as well as those of Greenleaf, Iosevich and Taylor \cite{GIT21,GIT22,GIT24} that apply to $\Phi$-configurations, require the user to verify some intricate calculations such as decay of the Fourier transform of the natural measure induced on the $\Phi$-configuration set or mapping properties of its associated Fourier integral operators. These are applicable on a case by case basis but seem far from establishing a dream result, where one takes any configuration of distances and can state a non-trivial threshold for it, by simply looking at the structure of the configuration. Finally, Chatzikonstantinou, Iosevich, Mkrtchyan and Pakianathan \cite{CIMP2021} were able to establish something resembling a dream theorem for a specific class of graphs called infinitesimally rigid graphs. While more graph based, it still required taking derivatives of a certain distance function, which can be computed on a case by case basis and is not easily determined by quickly glancing at the graph. A theme in many of these previous results is that fairly rigid configurations, such as simplices, or configurations amenable to inductive arguments can be handled. 

One of the main challenges so far is that anything with cycles is hard to handle, with only two results even existing on cycles \cite{GIP17,IMMM25} (where cycles were called necklaces in the first reference), and they only handle even length cycles. The main theorem of this paper establishes a significant step toward the dream result. By introducing the notion of \emph{$k$-admissibility}—a purely graph-theoretic condition depending only on the structure of the pinned point configuration— we get non-trivial thresholds for any configuration built out of distances in high enough dimensions (relative to the complexity of the graph itself). Importantly, our methods can handle cycles and even establish good dimensional thresholds for them. We are able to achieve all of this by identifying building blocks that work for any graph built out of distances.

A starting point towards our main theorem is a classic concept from graph theory on the degeneracy of a graph \cite{LW70}.

\begin{defn}
A graph $G$ is said to be $k$-degenerate if $k$ is the least number such that every induced subgraph of $G$ contains a vertex with $k$ or fewer neighbors.
\end{defn}

For a graph $G=(\mathcal{V},\mathcal{E})$, with vertex set $\mathcal{V}$ and edge set $\mathcal{E}$, Matula and Beck \cite{Matula} described an algorithm that can efficiently compute the degeneracy of a graph in $O(|\mathcal{V}|+|\mathcal{E}|)$ time. In short, the algorithm proceeds by repeatedly removing the vertex with the smallest degree and the degeneracy $k$ is given by the highest degree of any vertex at the time of its removal.

The notion of a $k$-degenerate graph already allows us to handle general graphs of distances. For a graph $G=(\mathcal{V},\mathcal{E})$ denote the $G$-distance set of a compact subset $E\subset\mathbb{R}^d$ by
$$ \Delta^{G}(E):= \lbrace (|x_{i}-x_{j}|)_{(v_i,v_j)\in\mathcal{E}} : x_1,\ldots,x_{|\mathcal{V}|}\in E \rbrace $$
where $\mathcal{V}=\{v_1,v_2,\dots ,v_{|\mathcal{V}|}\}$ and, for convenience, we represent the edges in $\mathcal{E}$ as ordered pairs $(v_i,v_j)$ with $i<j$. We say the graph $G$ is \emph{non-trivial} in this setting if $|\mathcal{E}|\geq 1$ so the $G$-distance set is not vacuous. Having at least one edge immediately implies that all non-trivial graphs we consider have degeneracy at least $1$. To take an example of a $G$-distance set, if $G=K_3$, the complete graph on $3$ vertices, we exactly recover the set of triangles, namely $\Delta^{K_3}(E)=\lbrace (|x_1 - x_2|,|x_1 - x_3|,|x_2 - x_3|) : x_1,x_2,x_3 \in E\rbrace$. 

\begin{thm}\label{thm: nopinthm}
Let $G=(\mathcal{V},\mathcal{E})$ be a non-trivial graph that is $k$-degenerate. Then for any compact set $E\subset\mathbb{R}^d$, $d>k$, with $\dim(E)>\frac{d+k}{2}$ the $|\mathcal{E}|$-dimensional Lebesgue measure of $\Delta^{G}(E)$ is positive.
\end{thm}

The degeneracy of a finite graph is always finite, so this result establishes a non-trivial dimensional threshold for a Falconer type result for any finite non-trivial distance graph in a high enough dimension. To the best of our knowledge, the notion of the degeneracy of a graph was first used in a setting of distance problems by Lyall and Magyar who established Euclidean Ramsey theory results for distance graphs \cite{LM20} under a further general position assumption using very different techniques, due to the completely different setting. If calculating the degeneracy of a graph seems daunting, it might be even easier to find the maximum vertex degree of a graph. The degeneracy of a graph is bounded above by the maximum vertex degree of the graph, as can immediately be seen by the algorithm of Matula and Beck. This immediately leads to the following Corollary.

\begin{cor}
Let $G=(\mathcal{V},\mathcal{E})$ be a non-trivial graph with maximum vertex degree $v_{\text{max}}$. Then for any compact set $E\subset\mathbb{R}^d$, $d > v_{\text{max}}$, with $\dim(E)>\frac{d+v_{\text{max}}}{2}$ the $|\mathcal{E}|$-dimensional Lebesgue measure of $\Delta^{G}(E)$ is positive.
\end{cor}

While this Corollary gives a very quick way to obtain a non-trivial dimensional threshold, we warn the reader that there can be a large discrepancy between the degeneracy of a graph and its maximum vertex degree. For example, a tree can have an arbitrarily large maximum vertex degree by simply piling up leafs at a common vertex while the degeneracy of a tree is always $1$. We note that in the setting of finite fields, such maximal vertex degree results for distance graphs were established by Iosevich and Parshall \cite{IP19} and extended to more general distances by Bright et. al. in \cite{BFHIJPS24}. 

Theorem \ref{thm: nopinthm} follows from an even more general framework that we describe in Section \ref{sec: mainresults}. In order to handle pinned configurations, we generalize the graph theory notion of $k$-degeneracy to what we call $k$-admissibility, which we define in detail and give many examples of in Section \ref{sec: admissibility}. Similar to the setting of $k$-degeneracy we develop an algorithm in Proposition \ref{prop:algorithmkadm} that quickly allows one to find the admissibility number of the graph. Theorem \ref{thm: firststructuralthm} is our main theorem for distances and implies Theorem \ref{thm: nopinthm}. Our method is even more general and can handle general edge vector functions $\Phi(\cdot,\cdot)$ in Theorem \ref{thm: secondstructuralthm}, conditional on having results on the appropriate building blocks. 

Beyond establishing the dream theorem that yields non-trivial thresholds in high enough dimension for any non-trivial configuration built out of distances, our theorem yields strong results for particular configurations, especially those that contain cycles. To showcase this let's define the set of $l$-cycles, $l\geq 3$, generated by a compact set $E\subset\mathbb{R}^d$ by
$$ \Delta^{l-cycle}(E) := \lbrace (|x_1 - x_2|, |x_2 - x_3|, \ldots, |x_{l-1} - x_{l}|, |x_{l} - x_{1}|) : x_1, x_2, \ldots, x_l \in E \rbrace \subset\mathbb{R}^l .$$
For even length cycles, which correspond to $l$ being even, Greenleaf, Iosevich and Pramanik in \cite{GIP17} showed non-empty interior of $\Delta^{l-cycle}(E)$, in the setting of constant gaps where all the lengths are equal, under the dimensional threshold $\dim(E)>\frac{d+3}{2}$, $d\geq 4$. We note they called the cycles necklaces but we use the more common graph theory name cycles here. Their result was very recently improved by Iosevich, Magyar, McDonald and McDonald \cite{IMMM25} where they extend the result for the $4$-cycle (still with constant gap) to $\dim(E)>\frac{53-\sqrt{337}}{12}$ when $d=3$. We emphasize that there are no results for odd cycles, other than the triangle, which corresponds to $l=3$. As an application of Theorem \ref{thm: nopinthm} we obtain a dimensional threshold $\dim(E)>\frac{d+2}{2}$, $d\geq 3$, for a compact set $E\subset\mathbb{R}^d$ to guarantee that $\Delta^{l-cycle}(E)$ the set of $l$-cycles, $l\geq 3$, has positive $l$-dimensional Lebesgue measure. Only in the case $l=3$, which corresponds to triangles, was this result known \cite{IPPS22}, while in finite fields a similar result existed using very different methods \cite{IJMD21}. We give a thorough exposition of a stronger multiple pin version of this result in Section \ref{sec: pinnednecklaces} and use this as an example to show a first step towards the underlying mechanism of the proof of Theorem \ref{thm: firststructuralthm}.

The building block, that drives our results for general pinned graphs of distances, are so called $k$-stars. Let $k\geq 1$ be an integer. By a $k$-star we will mean a graph in $(k+1)$ vertices in which a central vertex is connected by edges to any other vertex, and these are the only vertices in the graph. We will denote such a graph by $S_k$ (see the Figure \ref{fig: stargraph} for $S_7$).

\begin{figure}[h]
    \centering
    \begin{tikzpicture}[scale=1.5, 
    every node/.style={circle, draw, minimum size=6pt, inner sep=0pt},
    label distance=1mm]

% Central vertex
\node[fill=black, label=below:$v_8$] (C) at (0.2,0) {};

% Outer red vertices and edges
\foreach \i in {1,...,7} {
    \node[fill=magenta, label=above:$v_{\i}$] 
        (A\i) at ({cos(360/7*\i)}, {sin(360/7*\i)}) {};
    \draw (C) -- (A\i);
}

\end{tikzpicture}
    \caption{$S_7$, the $7$-star graph, with pins at each leaf.}
    \label{fig: stargraph}
\end{figure}
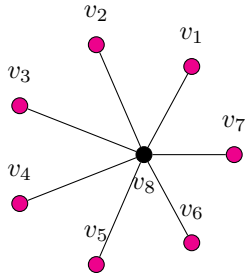

\begin{defn}\label{def: pinnedkstarsinE} Let $E\subset \R^d$ be a compact set.
The set of $k$-stars generated by $E$ with pins at $x_1,x_2,\dots,x_k$ is given by 
$$\Delta^{k-star}_{x_1,\dots,x_k}(E):=\{(|x_1-y|,|x_2-y|,\dots,|x_k-y|)\colon y\in E\}\subset \R^k.$$
\end{defn}

In \cite{IPPS22} the authors showed that if $E\subset \R^d$ is a compact set with $\dim(E)>\frac{d+k}{2}$ then there exist $x_1,x_2,\dots, x_k\in E$ such that $\Delta^{k-star}_{x_1,\dots,x_k}(E)$ has positive $k$-dimensional Lebesgue measure. In fact, they proved something even stronger, which we will recall below in Theorem \ref{thm:k starinIPPS}.

First, let us introduce a couple of definitions. Given a compact set $E\subset \R^d$ an $s$-Frostman measure $\mu$ on $E$ is a compactly supported (on $E$) finite Borel measure satisfying that 
$$\mu(B(x,r))\lesssim r^s,\, \text{ for all }r>0.$$
By Frostman's lemma \cite[Theorem 2.7]{Mattilabook2015}, such measures exist for all $0<s<\dim E$.

\begin{defn}\label{def: restrictedmeasure}
    Given a Frostman measure $\mu$ supported on a compact $E\subset \R^d$ and $F$ a Borel set in $\R^d$, denote by $\mu_F$ the restriction of $\mu$ to $F$. That is, $\mu_F(A)=\mu(A\cap F)$ for any Borel set $A\subset \R^d$.

\end{defn}

We are now ready to state a result proved in \cite{IPPS22}.
\begin{thm}[\cite{IPPS22}]\label{thm:k starinIPPS}

Let $E\subset \R^d$ be a compact set with $\dim(E)>\frac{d+k}{2}$. Take $\mu$ an $s$-Frostman measure on $E$, with $s>\frac{d+k}{2}$. 
    Let $E_1,E_2,\dots E_{k+1}$ be pairwise separated compact subsets of $E$ with $\mu(E_i)>0$ for all $i=1,2,\dots ,k+1$, then  for $\mu_{E_1}\times \mu_{E_2}\times \cdots \times\mu_{E_{k}}$ almost every $(x_1,x_2,\dots ,x_k)\in E_1\times E_2\times \dots \times E_k$, one has that the pushforward measure
    $$(d^{k-star}_{x_1,x_2,\dots ,x_k})_{*}(\mu_{E_{k+1}})\text{ belongs to } L^2(\R^k),$$
    where $d^{k-star}_{x_1,x_2,\dots ,x_k}(y):=(|x_1-y|,|x_2-y|,\dots ,|x_k-y|)$, and, consequently,
$$\mathcal{L}^k\left(\Delta^{k-star}_{x_1,\dots,x_k}(E_{k+1})\right)>0.$$
\end{thm}

In the statement above and later on in the paper, if $\nu$ is a measure in $\R^k$ we are using the standard abuse of notation $\nu\in L^p(\R^k)$ to mean that the measure $\nu$ is absolutely continuous with respect to Lebesgue measure of $\R^k$ and its Radon-Nikodym derivative is in $L^p(\R^k)$. Leveraging this result in an efficient way, as well as extending a Fubini style argument pioneered by Taylor and the third named author \cite{OT22} to the multiple pinned setting are some of the key ingredients that allow us to establish our main theorems. In the case of chains, our argument is essentially a \emph{backwards} version of the Fubini type argument developed in \cite{OT22}. However, for general graphs with cycles and multiple pins, it was unclear how the argument in \cite{OT22} would carry over. We resolve this issue by identifying the correct graph building strategy for our $k$-admissible graphs in this new setting and making use of the building blocks $k$-star. An interesting feature arises in our argument, where it is not enough to simply have abundance of the necessary building blocks, as for instance in the special case of cycles (see Section \ref{sec: pinnednecklaces}) but we rather need the full strength of the $L^2$ density results.

While our Fubini style argument is quite general and can be stated for quite general $\Phi(\cdot,\cdot)$ configurations we in general have no analogue of Theorem \ref{thm:k starinIPPS} for configurations other than the Euclidean distance. The proof in \cite{IPPS22} uses delicate projection type arguments that makes use of the Euclidean distance being realized through a dot product and thus it is even unclear how to extend their arguments to other distances, let alone other configurations. Establishing pinned $k$-star results is thus an open exciting problem for many configurations that immediately would lead to a broad set of general theorems using our Theorem \ref{thm: secondstructuralthm}.

\begin{rem}\label{rem:necessarykforkstars}
    One can observe that nontrivial results for pinned $k$-stars along the lines of Theorem \ref{thm:k starinIPPS} are only possible in dimensions $d>k$. Indeed, suppose that $E$ is a compact set in $\R^d$ such that there exists $x_1,x_2,\dots, x_k\in \R^d$  satisfying that $\mathcal{L}^k(\Delta^{k-star}_{x_1,\dots ,x_k}(E))>0$. Since the map $$\Phi_{x_1,\dots ,x_k}:E\rightarrow \R^k,\qquad x_{k+1}\mapsto(|x_1-x_{k+1}|,|x_2-x_{k+1}|,\dots, |x_k-x_{k+1}|)$$ is a Lipchitz map and $\Delta^{k-star}_{x_1,\dots ,x_k}(E)=\Phi_{x_1,x_2,\dots ,x_k}(E)$, it follows that $\dim(E)\geq k$ (and even $\mathcal{L}^k(E)>0$ in the case $k=d$). In particular, there is no nontrivial Hausdorff dimension threshold for obtaining a positive measure of doubly pinned $2$-chains in $\R^2$.
\end{rem}

In this paper the configuration functions $\Phi(\cdot,\cdot)$ we study always have two inputs, as they are applied to vertices connected by edges in our graphs. A natural question emerges about what happens in the hypergraph setting, where the configuration functions can have more than two inputs. In a simple case, where the output of the configuration function is just based on multiple two point configuration functions, such as returning the three distances determined by three points, then clearly one can still use the theory from this paper. In a sequel \cite{BFOPRA26} we study the particular example of areas of triangles. In certain cases we show how results from this paper can even apply in that setting, while in other cases we have a glimpse of a genuine hypergraph theory. While connected to the theme of this paper, the results are also interesting in their own right, with some of the base cases we establish extending known results in area configurations and relying on different methods from the current paper.

Here is an overview of the organization of the paper. In Section \ref{sec: mainresults} we discuss the setup and state the main results. In Section \ref{sec: admissibility} we define and give many examples of the notion of admissibility, the key graph theory notion we use. In Section \ref{sec: graphdismantling} we provide a key graph dismantling lemma. In Section \ref{sec: pinnednecklaces} we show how to prove our results for simpler cases of pinned chains and pinned cycles where we don't need the full strength of $L^2$ density results for the building blocks. Finally, in Section \ref{sec: proofofstructuraltheorems} we provide proof of our main structural theorems. 

\subsection*{Acknowledgement}
Y.O. is supported in part by NSF DMS-2142221 and NSF DMS-2055008. The authors would like to thank Alex Iosevich, Jill Pipher and Maya Sankar for many helpful and inspiring conversations on the topic of this paper.

\section{Setup and main results}\label{sec: mainresults}

Let $G$ be a simple graph with vertex set $\mathcal{V}$ and edge set 
$\mathcal{E}=\{e_1,\dots,e_{|\mathcal{E}|}\}$ and assume that $G$ is non-trivial, that is, $|\mathcal{E}|\geq 1$. 
For convenience, we represent each edge by an ordered pair $e=(u,v)$, even though the graph is undirected.

Let $\mathcal{P}\subset \mathcal{V}$ denote the subset of \emph{pinned vertices}, 
and assume that no two vertices in $\mathcal{P}$ are adjacent, that is, 
no edge in $\mathcal{E}$ connects two vertices of $\mathcal{P}$. 
Set $m:=|\mathcal{P}|$, where $0\le m\le |\mathcal{V}|$ 
(with $m=0$ corresponding to the unpinned case $\mathcal{P}=\emptyset$). 

We fix a canonical ordering of the vertices 
$v_1,\dots,v_{|\mathcal{V}|}$ so that $v_1,\dots,v_m$ are precisely 
the pinned vertices. 
Given a compact set $E\subset \mathbb{R}^d$, we consider 
\emph{edge vectors} generated by $E$, the graph $G$, 
and a continuous mapping
\[
\Phi=(\Phi_{e_1},\dots,\Phi_{e_{|\mathcal{E}|}}):\mathbb{R}^{d|\mathcal{V}|}\to 
\mathbb{R}^{|\mathcal{E}|}.
\]

Each component $\Phi_{e_i}$ is a two-variable function associated with 
the edge $e_i$; if $e_i$ connects the vertices $v_{i_1}$ and $v_{i_2}$, 
then $\Phi_{e_i}$ depends only on the corresponding coordinates 
$x_{i_1},x_{i_2}$ of 
$x=(x_1,\dots,x_{|\mathcal{V}|})\in \mathbb{R}^{d|\mathcal{V}|}$. We also assume that $\Phi_{e_i}$ is symmetric, that is, $\Phi_{e_i}(a,b)=\Phi_{e_i}(b,a)$, $a,b\in \R^d$.
We refer to such a map $\Phi$ as an \emph{edge vector function} for $G$.

Finally, for distinct points $x_1,\dots,x_m\in \mathbb{R}^d$, 
we define the \emph{$m$-pinned $(G,\Phi)$-set} of $E$ by
\[
\Delta^{G,\Phi}_{\mathcal{P},x_1,\dots,x_m}(E)
:=\bigl\{\Phi(x_1,\dots,x_m,y_{m+1},\dots,y_{|\mathcal{V}|})
:\; y_{m+1},\dots,y_{|\mathcal{V}|}\in E \text{ distinct}\bigr\}.
\]

In the particular case where $\Phi$ is chosen to be the Euclidean distance vector for which $\Phi_{e_i}(x,y)=|x-y|$ for all $i$, so each $\Phi_{e_i}$ maps a pair of vertices that are connected by an edge to their Euclidean distance, we often omit the function $\Phi$ from the notation. That was the case in  Definition \ref{def: pinnedkstarsinE}, which corresponds to the $k$-star graph pinned at all the vertices except the central one. In the case of the pinned $k$-star graph, we also omitted $\mathcal{P}$ (choice of locations to pin) from the notation, since it is implicitly understood that the unpinned vertex is the central one.

Other examples for an edge vector function $\Phi$ would be the dot product vector whose components $\Phi_{e_i}$ take dot products of pairs of vertices, or a different option is to replace dot product with the function that computes the area of triangles given by two vertices and the origin.

The main question that we would like to explore here is under what conditions the set $\Delta^{G,\Phi}_{\mathcal{P}, x_1,\cdots,x_m}(E)$ has positive $|\mathcal{E}|$-dimensional Lebesgue measure for some $x_1,x_2,\dots ,x_m\in E$. Since all the graphs we consider here are finite, we can fix a canonical order of edges in the definition of the function $\Phi$. And our result will be independent of the order we choose. In fact, we will always choose the order that is consistent with our definition of \emph{$k$-admissible} below. More remarks will be given later.

In order to avoid degeneracy issues, we oftentimes directly work with a slightly smaller $(G,\Phi)$-set where the points are chosen from pairwise separated subsets of $E$. Let $E_{m+1},\cdots, E_{|\mathcal{V}|}$ be a collection of such sets (usually chosen so that $\mu(E_j)>0$, where $\mu$ is Frostman measure in $E$). We have the following definition
\[
\begin{split}
&\Delta^{G,\Phi}_{\mathcal{P}, x_1,\cdots,x_m}(E_{m+1},\cdots, E_{|\mathcal{V}|})\\
:=&\{\Phi(x_1,\cdots, x_m, y_{m+1},\cdots, y_{|\mathcal{V}|}):\, y_{m+1}\in E_{m+1},\cdots, y_{|\mathcal{V}|}\in E_{|\mathcal{V}|}\}.
\end{split}
\]

\begin{atheorem}[Structural theorem for Euclidean distances]\label{thm: firststructuralthm}
Let $G$ be a graph with vertex set $\mathcal{V}$ and edge set $\mathcal{E}$. Let $\mathcal{P}\subset \mathcal{V}$ be such that no pair of vertices in $\mathcal{P}$ is connected by an edge in $\mathcal{E}$ and denote $m=|\mathcal{P}|$. Assume that $(G,\mathcal{P})$ is $k$-admissible and let $\Phi$ be the Euclidean distance vector function. Then, for $d> k$ and any compact set $E\subset \mathbb{R}^d$ with Hausdorff dimension larger than $\frac{d+k}{2}$, there are distinct points $x_1,\cdots, x_m\in E$ (and even positive $\Pi_{i=1}^{m}\mu_{E_i}$ measure worth of $m$-tuples $(x_1,x_2,\dots ,x_m)$ where $\mu_{E_i}$ are Frostman measures on separated subsets $E_i$ of $E$), such that the pinned $(G,\Phi)$-graph set $\Delta^{G,\Phi}_{\mathcal{P}, x_1,\cdots,x_m}(E)$ has positive $|\mathcal{E}|$-dimensional Lebesgue measure.

Moreover, when $k=1$ and $d\geq 2$, the threshold $\frac{d+1}{2}$ can be replaced with the sharpest currently known pinned Falconer distance thresholds, namely, $\dim(E)>\frac{5}{4}$ for $d=2$ and $\dim(E)>\frac{d}{2}+\frac{1}{4}-\frac{1}{8d+4}$ for $d\geq 3$.
\end{atheorem}

The notion of $k$-admissibility used in the theorem above will be introduced in Section \ref{sec: admissibility}. As an important corollary of the structural theorem above, and using the fact that cycles are $2$-admissible (see Subsection \ref{subsec: cyclesare2adm}), we get to the following result.

\begin{cor}[Corollary of the structural theorem for cycles]
 Let $l\geq 3$. Let $C_l$ be an $l$-cycle graph with $m\geq0$ pins listed in $\mathcal{P}$, such that no pair of vertices in $\mathcal{P}$ shares an edge. Then, if $E\subset \R^d$ is a compact set satisfying that $\dim(E)>\frac{d+2}{2}$, there are distinct points $x_1, x_2, \dots ,x_m$ such that one has positive measure for pinned $l$-cycles in the sense that
$$\mathcal{L}^l\left(\Delta_{\mathcal{P},x_1,x_2,\dots ,x_m}^{C_l}(E)\right)>0.$$
\end{cor}

In fact, if one is mainly interested in pinned chains and pinned cycles that are not triangles (with potentially multiple pins), the $L^2$ building blocks for $2$-stars provided in Theorem \ref{thm:k starinIPPS} from \cite{IPPS22} are not necessary, and it is enough to have the weaker positive measure conclusion of that theorem for $2$-stars. We expand on that in Section \ref{sec: pinnednecklaces} where we provide a self-contained proof of the threshold $(d+2)/2$ for pinned chains, as well as  for pinned cycles of length at least $4$. 

Theorem \ref{thm: firststructuralthm} above can be generalized to more general distance functions than the Euclidean distance, provided one has the right building blocks. We now make that precise. The notion of construction order used below will be introduced in Definition \ref{def: contructionorder} in Section \ref{sec: admissibility}.

\begin{defn}\label{def: suitableforstars}
    Let $\alpha\in (0,d)$ and let $\Phi=(\Phi_{1},\Phi_{2}\cdots, \Phi_{|\mathcal{E}|})$ be a continuous edge vector function for a pinned graph $(G,\mathcal{P}=(v_1,\dots ,v_m))$. Let $1\leq p\leq \infty$. We will say that $\Phi$ is $\alpha$-suitable for the stars in $(G,\mathcal{P})$ in $L^p$ sense if there is a construction order $\mathcal{O}=(v_1,v_2,\dots ,v_m,v_{m+1},\dots, v_l)$ such that for all $m+1\leq r\leq l=|\mathcal{V}|$, the following holds. Let $\eta_r$ be the set of indices of vertices of $G$ sharing an edge with $v_r$ and which appear earlier than $v_r$ in the construction order $\mathcal{O}$ and let $\epsilon_r:=|\eta_r|$ is back-degree of $v_r$ according to such ordering.
    Then, for any compact set $E$ with $\dim(E)>\alpha$, and any $\mu$ $s$-Frostman measure on $E$ where $s>\alpha$, if $\{E_{i}\}_{i\in \eta_r}$ are $\epsilon_r$ many separated compact subsets of $E$ with $\mu(E_i)>0$ for all $i\in \eta_r$ then  for $\Pi_{i\in \eta_r}\mu_{E_i}$ almost everywhere $(x_i)_{i\in \eta_r}\in \Pi_{i\in \eta_r} E_i$ one has 
    that 
    $$(d^{\epsilon_r-star,\Phi^r}_{(x_i)_{i\in \eta_r}} )_{*}(\mu_{E_{r}})\text{ belongs to } L^p(\R^{\epsilon_r}),$$
    where 
    $$d^{\epsilon_r-star,\Phi^r}_{(x_i)_{i\in \eta_r}}(y):=(\Phi_{(v_i,v_r)}(x_i,y))_{i\in \eta_r }.$$ In particular,
$$\mathcal{L}^{\epsilon_r}\left(\Delta^{\epsilon_r-star,\Phi^r}_{(x_i)_{i\in \eta_r}}(E_{r})\right)>0.$$
\end{defn}

\begin{atheorem}[Structural theorem for general edge vector functions]\label{thm: secondstructuralthm}
Let $G$ be a graph with vertex set $\mathcal{V}$ and edge set $\mathcal{E}$. Let $\mathcal{P}\subset \mathcal{V}$ be such that no pair of vertices in $\mathcal{P}$ is connected by an edge in $\mathcal{E}$ and denote $m=|\mathcal{P}|$. % Let $(G,\mathcal{P})$ be a $k$-admissible pinned graph with $1\leq k\leq |\mathcal{E}|$.
Let $\alpha\in (0,d)$ and $\Phi=(\Phi_{e_1},\cdots, \Phi_{e_{|\mathcal{E}|}})$ be any given continuous edge vector function of $G$ satisfying that $\Phi$ is $\alpha$-suitable for the stars in $G$ in $L^p$ sense, for some $1\leq p\leq \infty$.
Then, for any compact set $E\subset \mathbb{R}^d$ with $\dim(E)>\alpha$, there exist $x_1,\cdots, x_m\in E$ such that the pinned $(G,\Phi)$-graph set $\Delta^{G,\Phi}_{\mathcal{P},x_1,\cdots, x_m}(E)$ has positive $|\mathcal{E}|$-dimensional Lebesgue measure.
\end{atheorem}

\section{Notion of admissibility and some key examples}\label{sec: admissibility}

\subsection{Allowed moves: definition of $k$-admissible}
Following the notation above, let $G=(\mathcal V,\mathcal E)$ be a simple graph with pins $\mathcal{P}\subset\mathcal V$ and $|\mathcal{P}|=m$. List the pins as $v_1,\dots,v_m$. We say that the pinned graph $(G,\mathcal{P})$ is $k$-admissible, $k\in\mathbb N$, if $G$ can be constructed in the steps below.

\begin{enumerate}
\item Add the pinned vertices $v_1,\cdots, v_m$ (this step is skipped if there are no pinned vertices).
\item Choose a vertex $v_{m+1}\in \mathcal{V}\backslash\mathcal{P}$ 
 such that the number of edges in $\mathcal{E}$ connecting $v_{m+1}$ to the previously added vertices is at most $k$. Add $v_{m+1}$ and all such edges.
\item Repeat the above step to add another vertex $v_{m+2}\in {\mathcal{V}}\backslash (\mathcal{P}\cup\{v_{m+1}\})$ and all the edges in $\mathcal{E}$ connecting $v_{m+2}$ to the previously added vertices (at most $k$ such edges). Continue the process.
\item Continue until all vertices in $\mathcal{V}$ are added (and hence all edges in $\mathcal{E}$ are revealed).
\end{enumerate}

The next definition allows us to encode $k$-admissibility succinctly.

\begin{defn}[Construction order and back-degree]\label{def: contructionorder}
We say that an ordering $v_1,\dots,v_n$ of $\mathcal V$ with $v_1,\dots,v_m\in\mathcal{P}$ is a \emph{construction order} for $(G,\mathcal{P})$.  
For $i>m$, define the \emph{back-degree}
\[
d_{<i}(v_i):=\big|\{\,u\in\{v_1,\dots,v_{i-1}\}:\ (u,v_i)\in\mathcal E\,\}\big|.
\]
Operationally, when $v_i$ is added, we reveal all edges from $v_i$ to the already added set; edges whose endpoints are both not yet added are revealed when the later endpoint appears.
\end{defn}

\begin{defn}[$k$-admissible; admissibility number]
For $k\in\mathbb N_0$, a pinned graph $(G,\mathcal{P})$ is \emph{$k$-admissible} if there exists a construction order with $d_{<i}(v_i)\le k$ for all $i>m$. We may refer to such a construction order as a \textit{$k$-admissible construction order}. The \emph{admissibility number} is defined as
\[
\kappa(G,\mathcal{P}):=\min\{\,k\in \N_0:\ (G,\mathcal{P})\ \text{is $k$-admissible}\,\}.
\]
\end{defn}

Trivially,
\[
0\ \le\ \kappa(G,\mathcal{P})\ \le\ \Delta(G)\ \le\ \min\{\,|\mathcal V|-1,\ |\mathcal E|\,\},
\]
 where $\Delta(G)$ is the maximal degree for a vertex in $G$.

 A $k$-star $S_k$ with pins at all the leaves is $k$-admissible since the central vertex has the maximal degree $k$. In fact, $k$ is the admissibility number for such pinned graph. In the next subsection, we give various examples of pinned graphs and their corresponding admissibility.

\subsection{Admissibility for some key examples}

To minimize the admissibility $k$, we prioritize adding vertices that have the largest relative degree to those that have already been added. (In other words, after step 1, we should first add the $v_{m+1}$ that has the most edges connected to $v_1,\cdots, v_m$, and repeat this at each step). To understand why that is advantageous, consider the 
example where $\mathcal{P}=\{v_1,v_2\}\subset\mathcal{V}=\{v_1,v_2,v_3,v_4\}$, and the edge set is $\mathcal{E}=\{(v_1,v_3), (v_1,v_4), (v_2,v_3),(v_3,v_4)\}$.  
We start with the pinned vertices $\{v_1,v_2\}$ and no edges. If $v_3$ is added first, its back-degree is $2$, and then $v_4$ follows with back-degree $2$, giving $\kappa(G,\mathcal{P})\le 2$.  
If instead $v_4$ is added first (back-degree $1$), then $v_3$ comes with back-degree $3$, so we only get $\kappa(G,\mathcal{P})\le 3$.  
Thus prioritizing the vertex with larger current back-degree leads to a better bound.

\subsubsection{Complete graph in $(k+1)$ vertices}    
     
     For a complete graph in $(k+1)$ vertices (i.e, every pair of points is connected by one edge), we can pin at most one vertex (since we are requiring the pins not to be connected by an edge). Independent of the location of the pin, such a graph is $k$-admissible, and $k$ is its admissibility number (since the last added unpinned vertex in the admissibility construction will always be connected to the previous $k$-vertices). Applying Theorem \ref{thm: firststructuralthm}, one recovers the result for pinned $k$-simplices in \cite[Theorem 1.9]{IPPS22}, namely that the distance set 
     $\{(|x_i-x_j|)_{1\leq i<j\leq k+1}\colon x_2,x_3,\dots ,x_{k}\in E \}$
     has positive $k+1\choose 2$-dimensional Lebesgue measure for some $x_1\in E$ if $\dim(E)>\frac{d+k}{2}$. For $ k=2$, this reduces to pinned triangles.
      See Figure \ref{fig: pinnedtriangle} and Figure \ref{fig: 4simplex} for $k=2$ and $k=4$, respectively.

\begin{minipage}[t]{0.5\textwidth}

\centering
\begin{tikzpicture}[scale=1.5, 
    every node/.style={circle, draw, minimum size=6pt, inner sep=0pt},
    label distance=1mm]

% Coordinates of an equilateral triangle rotated by 30 degrees
\coordinate (A) at ({cos(30)}, {sin(30)});
\coordinate (B) at ({cos(150)}, {sin(150)});
\coordinate (C) at ({cos(270)}, {sin(270)});

% Draw edges
\draw (A) -- (B) -- (C) -- cycle;
% Draw vertices
\node[fill=magenta, label=above left:\(v_1\)] at (A) {};
\node[fill=black, label=above right:\(v_2\)] at (B) {};
\node[fill=black, label=below:\(v_3\)] at (C) {};

\end{tikzpicture}

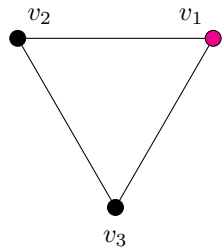
\captionof{figure}{Pinned triangle.}
\label{fig: pinnedtriangle}
  \label{fig:pinnedtriangle}
  
\end{minipage}%
\begin{minipage}[t]{0.5\textwidth}

\begin{center}
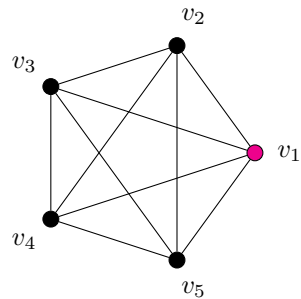

\begin{tikzpicture}[scale=1.5,
  every node/.style={circle, draw, minimum size=6pt, inner sep=0pt},
  label distance=1.5mm]

% Store coordinates of the 5 vertices
\foreach \i in {1,...,5} {
    \pgfmathsetmacro{\angle}{360/5*(\i-1)}
    \pgfmathsetmacro{\x}{cos(\angle)}
    \pgfmathsetmacro{\y}{sin(\angle)}
    
    \ifnum\i=1
        \node[fill=magenta, label={\angle:\(v_{\i}\)}] (V\i) at (\x,\y) {};
    \else
        \node[fill=black, label={\angle:\(v_{\i}\)}] (V\i) at (\x,\y) {};
    \fi
}

% Draw all edges (complete graph)
\foreach \i in {1,...,5} {
    \foreach \j in {\i,...,5} {
        \ifnum\i<\j
            \draw (V\i) -- (V\j);
        \fi
    }
}

\end{tikzpicture}
\end{center}

\captionof{figure}{Pinned complete graph in $5$ vertices.}
\label{fig: 4simplex}
\end{minipage}

    \medskip
    
     Note that the discussion above implies that for general graphs, even if we pin at just one vertex, arbitrary admissibility can be realized in the class of simple finite graphs. Moreover, for this example, the admissibility agrees with the maximal vertex degree $\Delta(G)$.

\subsubsection{Chains}

   Take $\text{Ch}_n$ to be a chain graph on $n$ vertices, with pins in $\mathcal{P}$, a collection of $m$ non-adjacent vertices of $\text{Ch}_n$. For $m=1$ (so also for $m=0$) it is easy to see that  $(\text{Ch}_n,\mathcal{P})$ is $1$-admissible, while for any $m\geq 2$, $(\text{Ch}_n,\mathcal{P})$ is $2$-admissible, which can be shown by induction on the number of unpinned vertices in the chain.

    \begin{center}
\begin{tikzpicture}[scale=1.2,
  vertex/.style={circle, draw, minimum size=6pt, inner sep=0pt}, % for vertices only
  label distance=1mm]\label{fig: multiplypinnedchain}

% Define vertex positions and draw them
\foreach \i in {1,...,6} {
    \pgfmathsetmacro{\x}{\i}
    \pgfmathsetmacro{\y}{mod(\i,2)}  % alternate up/down
    \pgfmathsetmacro{\yy}{\y * 1.0}  % vertical displacement

    \ifnum\i=1
        \node[vertex, fill=magenta, label=above:\(v_{\i}\)] (V\i) at (\x,\yy) {};
    \else\ifnum\i=3
        \node[vertex, fill=magenta, label=above:\(v_{\i}\)] (V\i) at (\x,\yy) {};
        \else\ifnum\i=5
        \node[vertex, fill=black, label=above:\(v_{\i}\)] (V\i) at (\x,\yy) {};
    \else\ifnum\i=6
        \node[vertex, fill=magenta, label=below:\(v_{\i}\)] (V\i) at (\x,\yy) {};
    \else
        \node[vertex, fill=black, label=below:\(v_{\i}\)] (V\i) at (\x,\yy) {};
    \fi\fi\fi\fi
}

% Draw edges with plain text labels e_1 to e_5
\foreach \i [evaluate=\i as \j using int(\i+1),
            evaluate=\i as \label using int(\i)] in {1,...,5} {
    \draw (V\i) -- (V\j)
        node[midway, above, font=\small] {\(e_{\label}\)};
}

\end{tikzpicture}
\captionof{figure}{Example of multiply pinned chain.}
\end{center}

To check the $2$-admissibility of the example in the figure above, where $\mathcal{P}=\{v_1,v_3,v_6\}$ (vertices in magenta), one can start with $v_1,v_3,v_6$, then add $v_2$ and the two edges $e_1,e_2$, then $v_4$ and edge $e_3$, and finally $v_5$ and its two edges $e_4,e_5$.

     \subsubsection{ Trees}  Let $\mathcal{T}$ be a tree with at least two vertices, with pins in $\mathcal{P}$, a collection of $m$ non-adjacent vertices in the tree. Then such multiple-pinned graph is $m$-admissible. By Proposition \ref{prop: adding pins}, it would be enough to check it for $m=1$, i.e, only one pin. Since any tree (with at least two vertices) pinned at a single location will always have an unpinned leaf, the proof for $m=1$ is a simple induction on the number of vertices of the tree. 

    Observe that $m$ might not be optimal for certain trees with $m$ pins. In fact, for any  $2\leq k\leq m$ there are trees with $m$ pinned vertices whose admissibility number is $k$. One can achieve that by adding $k$ pinned leaves to the unpinned endpoint of an appropriate chain with $m-k$ pins. For example for $m=5$ and $k=3$ one can consider the tree in Figure \ref{fig:3admissible5pinnedtree}.

%\begin{center}
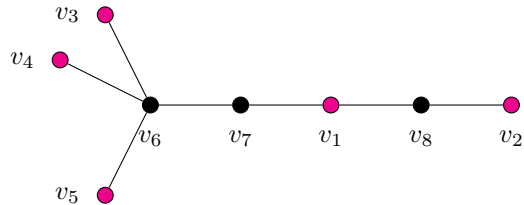
\begin{figure}[h]
\centering
\begin{tikzpicture}[scale=1.2,
  vertex/.style={circle, draw, minimum size=6pt, inner sep=0pt},
  label distance=1mm]

% ====== Main Chain ======
\node[vertex, fill=black, label=below:\(v_6\)] at (1,0) (V1) {};
\node[vertex, fill=black, label=below:\(v_7\)] at (2,0) (V2) {};
\node[vertex, fill=magenta, label=below:\(v_1\)] at (3,0) (V3) {};
\node[vertex, fill=black, label=below:\(v_8\)] at (4,0) (V4) {};
\node[vertex, fill=magenta, label=below:\(v_2\)] at (5,0) (V6) {}; % skipping v5

% Connect the chain
\draw (V1) -- (V2);
\draw (V2) -- (V3);
\draw (V3) -- (V4);
\draw (V4) -- (V6);

% ====== magenta Leaves connected to V1 ======
\node[vertex, fill=magenta, label=left:\(v_3\)] at (0.5, 1) (V7) {};
\node[vertex, fill=magenta, label=left:\(v_4\)] at (0, 0.5) (V8) {};
\node[vertex, fill=magenta, label=left:\(v_5\)] at (0.5, -1) (V9) {};

\draw (V1) -- (V7);
\draw (V1) -- (V8);
\draw (V1) -- (V9);

\end{tikzpicture}

\caption{Pinned tree with $5$ pins and admissibility number $3$.}
\label{fig:3admissible5pinnedtree}
\end{figure}
%\end{center}

%\captionof{figure}{Pinned tree with $5$ pins and admissibility number $3$.}
%\label{fig: 3admissible5pinnedtree}

\subsubsection{Multiply-pinned cycles}\label{subsec: cyclesare2adm} Figures \ref{fig: 8cycle} and \ref{fig: 7cycle} illustrate a couple of cycle graphs. If the cycle has an even number of vertices, we can pin at most half of them, while in the case of $|\mathcal{V}|=n$ is odd, then we can pin at most $(n-1)/2$ non-adjacent vertices. In either case, such pinned graphs are $2$-admissible since every vertex has degree $2$. Note that $2$ is their actual admissibility number since they are not $1$-admissible.

\begin{minipage}[t]{0.5\textwidth}
\begin{center}
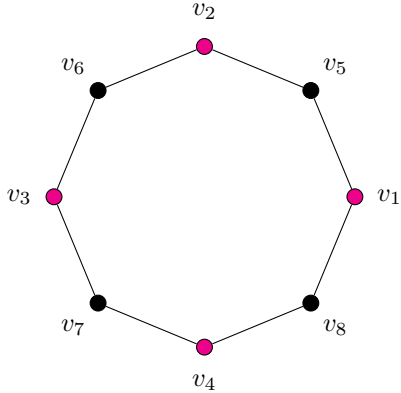

% First TikZ picture
\begin{tikzpicture}[scale=2,
    every node/.style={circle, draw, minimum size=6pt, inner sep=0pt},
    label distance=1.5mm]

% Loop to draw 8 vertices
\foreach \i in {1,...,8} {
    % Compute angle and coordinates
    \pgfmathsetmacro{\angle}{360/8*(\i-1)}
    \pgfmathsetmacro{\x}{cos(\angle)}
    \pgfmathsetmacro{\y}{sin(\angle)}
    \pgfmathsetmacro{\labelangle}{\angle}
    
    % Determine color and label index
    \ifodd\i
        \pgfmathtruncatemacro{\labelnum}{(\i+1)/2} % v1 to v4 for magenta
        \node[fill=magenta, label={\labelangle:\(v_{\labelnum}\)}] (V\i) at (\x,\y) {};
    \else
        \pgfmathtruncatemacro{\labelnum}{\i/2 + 4} % v5 to v8 for black
        \node[fill=black, label={\labelangle:\(v_{\labelnum}\)}] (V\i) at (\x,\y) {};
    \fi
}

% Draw edges
\foreach \i [evaluate=\i as \j using {int(mod(\i,8)+1)}] in {1,...,8} {
    \draw (V\i) -- (V\j);
}

\end{tikzpicture}
\captionof{figure}{$4$-pinned cycle with $8$ vertices.}
\label{fig: 8cycle}
\end{center}
\end{minipage}%
\begin{minipage}[t]{0.5\textwidth}
\begin{center}
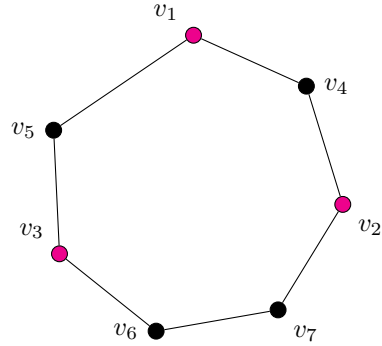

% Second TikZ picture

\begin{tikzpicture}[scale=2,
  vertex/.style={circle, draw, minimum size=6pt, inner sep=0pt}]

% Vertex placement with swapped labels and adjusted label positions
\node[vertex, fill=magenta,   label=above left:\(v_1\)] (V1) at (90:1) {};
\node[vertex, fill=black, label=right:\(v_4\)]     (V2) at (41.4:1) {};
\node[vertex, fill=magenta,   label=below right:\(v_2\)] (V3) at (-7.2:1) {};
\node[vertex, fill=black, label=below right:\(v_7\)] (V4) at (-55.8:1) {}; % label moved to below right
\node[vertex, fill=black, label=left:\(v_6\)]      (V5) at (-104.4:1) {};
\node[vertex, fill=magenta,   label=above left:\(v_3\)] (V6) at (-153:1) {};
\node[vertex, fill=black, label=left:\(v_5\)]      (V7) at (-201.6:1) {}; % label moved to left

% Draw edges in cycle order
\draw (V1) -- (V2) -- (V3) -- (V4) -- (V5) -- (V6) -- (V7) -- (V1);

\end{tikzpicture}
\captionof{figure}{$3$-pinned cycle with $7$ vertices.}
\label{fig: 7cycle}
\end{center}
\end{minipage}

\subsubsection{Trees of triangles}

In what follows, we consider a graph that is a tree of triangles. The vertices in magenta represent three different choices for the set of pins, and for which case we indicate the admissibility number of the resulting pinned graph.

\begin{minipage}{0.33\textwidth}
    \begin{center}
\begin{tikzpicture}[
  scale=1.5,
  vertex/.style={circle, fill=black, inner sep=1.5pt},
  redvertex/.style={circle, fill=magenta, inner sep=1.5pt},
  every label/.style={font=\small, inner sep=1pt}
]

  % Constants for equilateral triangles
  \def\L{1}             
  \def\H{0.8660254038}  

  % Coordinates (unchanged)
  \coordinate (v5) at (0, 0);        
  \coordinate (v1) at (-0.5*\L, \H); 
  \coordinate (v7) at (-\L, 0);      
  \coordinate (v8) at (0.5*\L, \H);  
  \coordinate (v2) at (\L, 0);       
  \coordinate (v6) at (-0.5*\L, -\H);
  \coordinate (v3) at (0.5*\L, -\H); 
  \coordinate (v9) at (-1*\L, -2*\H);
  \coordinate (v4) at (0, -2*\H);    

  % Draw triangles
  \draw[thick] (v5) -- (v1) -- (v7) -- cycle;
  \draw[thick] (v5) -- (v8) -- (v2) -- cycle;
  \draw[thick] (v5) -- (v6) -- (v3) -- cycle;
  \draw[thick] (v6) -- (v9) -- (v4) -- cycle;

  % Draw vertices with colors
  \node[vertex] at (v5) {};
  \node[redvertex] at (v1) {};
  \node[vertex] at (v7) {};
  \node[vertex] at (v8) {};
  \node[redvertex] at (v2) {};
  \node[vertex] at (v6) {};
  \node[redvertex] at (v3) {};
  \node[vertex] at (v9) {};
  \node[redvertex] at (v4) {};

  % Labels with new names
  \node[label=below right:\(v_5\)] at (v5) {};
  \node[label=above left:\(v_1\)] at (v1) {};
  \node[label=left:\(v_7\)] at (v7) {};
  \node[label=above right:\(v_8\)] at (v8) {};
  \node[label=right:\(v_2\)] at (v2) {};
  \node[label=left:\(v_6\)] at (v6) {};
  \node[label=below right:\(v_3\)] at (v3) {};
  \node[label=below left:\(v_9\)] at (v9) {};
  \node[label=below right:\(v_4\)] at (v4) {};

\end{tikzpicture}
\captionof{figure}{Admissibility number is $3$}
\end{center}

\end{minipage}\begin{minipage}{0.33\textwidth}
\begin{center}
\begin{tikzpicture}[
  scale=1.5,
  vertex/.style={circle, fill=black, inner sep=1.5pt},
  redvertex/.style={circle, fill=magenta, inner sep=1.5pt},
  every label/.style={font=\small, inner sep=1pt}
]

  % Constants for equilateral triangles
  \def\L{1}             
  \def\H{0.8660254038}  

  % Coordinates 
  \coordinate (v4) at (0, 0);        
  \coordinate (v6) at (-0.5*\L, \H); 
  \coordinate (v7) at (-\L, 0);      
  \coordinate (v8) at (0.5*\L, \H);  
  \coordinate (v1) at (\L, 0);       
  \coordinate (v5) at (-0.5*\L, -\H);
  \coordinate (v2) at (0.5*\L, -\H); 
  \coordinate (v9) at (-1*\L, -2*\H);
  \coordinate (v3) at (0, -2*\H);    

  % Draw triangles
  \draw[thick] (v4) -- (v6) -- (v7) -- cycle;
  \draw[thick] (v4) -- (v8) -- (v1) -- cycle;
  \draw[thick] (v4) -- (v5) -- (v2) -- cycle;
  \draw[thick] (v5) -- (v9) -- (v3) -- cycle;

  % Draw vertices with colors
  \node[vertex] at (v5) {};
  \node[redvertex] at (v1) {};
  \node[vertex] at (v7) {};
  \node[vertex] at (v8) {};
  \node[redvertex] at (v2) {};
  \node[vertex] at (v6) {};
  \node[redvertex] at (v3) {};
  \node[vertex] at (v9) {};
  \node[vertex] at (v4) {};

  % Labels with new names
  \node[label=left:\(v_5\)] at (v5) {};
  \node[label=right:\(v_1\)] at (v1) {};
  \node[label=left:\(v_7\)] at (v7) {};
  \node[label=above right:\(v_8\)] at (v8) {};
  \node[label=below right:\(v_2\)] at (v2) {};
  \node[label=above left:\(v_6\)] at (v6) {};
  \node[label=below right:\(v_3\)] at (v3) {};
  \node[label=below left:\(v_9\)] at (v9) {};
  \node[label=below right:\(v_4\)] at (v4) {};

\end{tikzpicture}
\end{center}
\captionof{figure}{\small{Admissibility number is $3$}}
\end{minipage}\begin{minipage}{0.33\textwidth}
    \begin{center}
\begin{tikzpicture}[
  scale=1.5,
  vertex/.style={circle, fill=black, inner sep=1.5pt},
  redvertex/.style={circle, fill=magenta, inner sep=1.5pt},
  every label/.style={font=\small, inner sep=1pt}
]

  % Constants for equilateral triangles
  \def\L{1}             
  \def\H{0.8660254038}  

  % Coordinates (unchanged)
  \coordinate (v4) at (0, 0);        
  \coordinate (v1) at (-0.5*\L, \H); 
  \coordinate (v7) at (-\L, 0);      
  \coordinate (v8) at (0.5*\L, \H);  
  \coordinate (v2) at (\L, 0);       
  \coordinate (v5) at (-0.5*\L, -\H);
  \coordinate (v6) at (0.5*\L, -\H); 
  \coordinate (v9) at (-1*\L, -2*\H);
  \coordinate (v3) at (0, -2*\H);    

  % Draw triangles
  \draw[thick] (v4) -- (v1) -- (v7) -- cycle;
  \draw[thick] (v4) -- (v8) -- (v2) -- cycle;
  \draw[thick] (v4) -- (v5) -- (v6) -- cycle;
  \draw[thick] (v5) -- (v9) -- (v3) -- cycle;

  % Draw vertices with colors
  \node[vertex] at (v5) {};
  \node[redvertex] at (v1) {};
  \node[vertex] at (v7) {};
  \node[vertex] at (v8) {};
  \node[redvertex] at (v2) {};
  \node[vertex] at (v6) {};
  \node[redvertex] at (v3) {};
  \node[vertex] at (v9) {};
  \node[vertex] at (v4) {};

  % Labels with new names
  \node[label=left:\(v_5\)] at (v5) {};
  \node[label=above left:\(v_1\)] at (v1) {};
  \node[label=left:\(v_7\)] at (v7) {};
  \node[label=above right:\(v_8\)] at (v8) {};
  \node[label=right:\(v_2\)] at (v2) {};
  \node[label=below right:\(v_6\)] at (v6) {};
  \node[label=below right:\(v_3\)] at (v3) {};
  \node[label=below left:\(v_9\)] at (v9) {};
  \node[label=below right:\(v_4\)] at (v4) {};

\end{tikzpicture}

\end{center}
\captionof{figure}{\small{Admissibility number is $2$}}
\end{minipage}

In the figures above, we have intentionally labeled the vertices in a way that suggests a possible order for adding vertices and edges to satisfy the claimed admissibility in each case. For example, we can check that the pinned graph on the  leftmost position is $3$-admissible by following the steps:
\begin{itemize}
    \item Add pinned vertices $v_1,v_2,v_3,v_4$;
    \item Add vertex $v_5$ and the $3$ edges connecting it to $v_1,v_2,v_3$ respectively;
    \item Add $v_6$ and the $3$ edges connecting it to $v_3,v_4$ and $v_5$;
    \item Add $v_7$ and the $2$ edges connecting it to $v_1,v_5$;
    \item Add $v_8$ and $v_9$ in a similar way as $v_7$ ($2$ edges for each of them).
\end{itemize}

\subsubsection{Double banana}

The following graph will be referred to as the double banana graph. 

\begin{center}
\begin{tikzpicture}[scale=1.5]

% --- Coordinates ---
\coordinate (A) at (0,0.3);        % v1
\coordinate (B) at (1,0.3);        % v2
\coordinate (C) at (0.6,0.6);    % v3
\coordinate (F) at (2,0.3);        % v6
\coordinate (G) at (3,0.3);        % v7
\coordinate (H) at (2.4,0.6);    % v8
\coordinate (D) at (1.5,1.6);    % v4
\coordinate (E) at (1.5,-0.9);   % v5

% --- Edges for Banana 1 ---
\draw[thick] (A) -- (B) -- (C) -- cycle;
\draw[thick] (A) -- (D); \draw[thick] (B) -- (D); \draw[thick] (C) -- (D);
\draw[thick] (A) -- (E); \draw[thick] (B) -- (E); \draw[thick] (C) -- (E);

% --- Edges for Banana 2 ---
\draw[thick] (F) -- (G) -- (H) -- cycle;
\draw[thick] (F) -- (D); \draw[thick] (G) -- (D); \draw[thick] (H) -- (D);
\draw[thick] (F) -- (E); \draw[thick] (G) -- (E); \draw[thick] (H) -- (E);

% --- Vertices (dots) ---
\fill[magenta]  (A) circle (1.5pt);   % v1 in blue
\fill[black] (B) circle (1.5pt);   % v2
\fill[black] (C) circle (1.5pt);   % v3
\fill[black] (D) circle (1.5pt);   % v4
\fill[black] (E) circle (1.5pt);   % v5
\fill[black] (F) circle (1.5pt);   % v6
\fill[black] (G) circle (1.5pt);   % v7
\fill[black] (H) circle (1.5pt);   % v8

% --- Labels (manually placed) ---
\node at (-0.2, 0.3) {$v_1$};
\node at (1.2, 0.3) {$v_2$};
\node at (0.8, 0.65) {$v_3$};
\node at (1.5, 1.75) {$v_4$};
\node at (1.5, -1.05) {$v_5$};
\node at (1.8, 0.3) {$v_6$};
\node at (3.2, 0.3) {$v_7$};
\node at (2.2, 0.65) {$v_8$};

\end{tikzpicture}

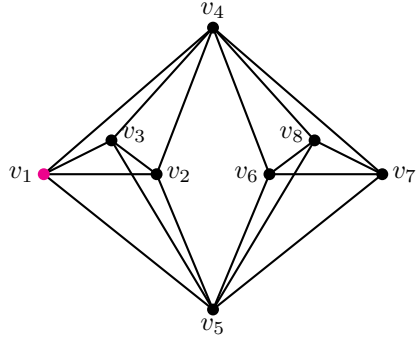
\captionof{figure}{The double banana graph.}
\end{center}

If we pin the vertex $v_1$, then adding vertices according to our labeling, we can see that the pinned graph is $4$-admissible. In fact, for any choice of the pinned vertex, the double banana graph is $4$-admissible, as the interested reader can easily check. By Theorem~\ref{thm: firststructuralthm}, we therefore obtain that this graph induces pinned distance sets  of positive measure in $\R^{18}$ provided that the compact set $E\subset \R^d$ where we sample the vertices has Hausdorff dimension strictly larger than $(d+4)/2$.

Next, suppose we delete the edge connecting $v_2$ and $v_4$, as in the figure below. Sampling the vertices of the graph in a set $E\subset \R^d$ we get a distance set in $\R^{17}$. For this example, we observe that the location of the pin may affect the admissibility number, so the choice of the pin matters in general. A pin at $v_7$ produces a $3$-admissible pinned graph (for instance, by adding vertices in the order $v_6,v_8,v_4,v_5,v_1,v_3,v_2$ in the admissibility definition, or by using Proposition \ref{prop:algorithmkadm}, since successively deleting unpinned vertices of degree at most $3$ and their corresponding edges eventually deletes all unpinned vertices), but the graph on the left with a pin at $v_1$ is not $3$-admissible.

\begin{minipage}{0.5\textwidth}
\centering
\begin{tikzpicture}[scale=1.5]

% --- Coordinates ---
\coordinate (A) at (0,0.3);        % v1
\coordinate (B) at (1,0.3);        % v2
\coordinate (C) at (0.6,0.6);    % v3
\coordinate (F) at (2,0.3);        % v6
\coordinate (G) at (3,0.3);        % v7
\coordinate (H) at (2.4,0.6);    % v8
\coordinate (D) at (1.5,1.6);    % v4
\coordinate (E) at (1.5,-0.9);   % v5

% --- Edges for Banana 1 ---
\draw[thick] (A) -- (B) -- (C) -- cycle;
\draw[thick] (A) -- (D); 
%\draw[thick] (B) -- (D); 
\draw[thick] (C) -- (D);
\draw[thick] (A) -- (E); \draw[thick] (B) -- (E); \draw[thick] (C) -- (E);

% --- Edges for Banana 2 ---
\draw[thick] (F) -- (G) -- (H) -- cycle;
\draw[thick] (F) -- (D); \draw[thick] (G) -- (D); \draw[thick] (H) -- (D);
\draw[thick] (F) -- (E); \draw[thick] (G) -- (E); \draw[thick] (H) -- (E);

% --- Vertices (dots) ---
\fill[magenta]  (A) circle (1.5pt);   % v1 in blue
\fill[black] (B) circle (1.5pt);   % v2
\fill[black] (C) circle (1.5pt);   % v3
\fill[black] (D) circle (1.5pt);   % v4
\fill[black] (E) circle (1.5pt);   % v5
\fill[black] (F) circle (1.5pt);   % v6
\fill[black] (G) circle (1.5pt);   % v7
\fill[black] (H) circle (1.5pt);   % v8

% --- Labels (manually placed) ---
\node at (-0.2, 0.3) {$v_1$};
\node at (1.2, 0.3) {$v_2$};
\node at (0.8, 0.65) {$v_3$};
\node at (1.5, 1.75) {$v_4$};
\node at (1.5, -1.05) {$v_5$};
\node at (1.8, 0.3) {$v_6$};
\node at (3.2, 0.3) {$v_7$};
\node at (2.2, 0.65) {$v_8$};

\end{tikzpicture}
\captionof{figure}{A $4$-admissible pinned graph that is not $3$-admissible.}

\end{minipage}\begin{minipage}{0.5\textwidth}
\centering
\begin{tikzpicture}[scale=1.5]

% --- Coordinates ---
\coordinate (A) at (0,0.3);        % v1
\coordinate (B) at (1,0.3);        % v2
\coordinate (C) at (0.6,0.6);    % v3
\coordinate (F) at (2,0.3);        % v6
\coordinate (G) at (3,0.3);        % v7
\coordinate (H) at (2.4,0.6);    % v8
\coordinate (D) at (1.5,1.6);    % v4
\coordinate (E) at (1.5,-0.9);   % v5

% --- Edges for Banana 1 ---
\draw[thick] (A) -- (B) -- (C) -- cycle;
\draw[thick] (A) -- (D); 
%\draw[thick] (B) -- (D); 
\draw[thick] (C) -- (D);
\draw[thick] (A) -- (E); \draw[thick] (B) -- (E); \draw[thick] (C) -- (E);

% --- Edges for Banana 2 ---
\draw[thick] (F) -- (G) -- (H) -- cycle;
\draw[thick] (F) -- (D); \draw[thick] (G) -- (D); \draw[thick] (H) -- (D);
\draw[thick] (F) -- (E); \draw[thick] (G) -- (E); \draw[thick] (H) -- (E);

% --- Vertices (dots) ---
\fill[black]  (A) circle (1.5pt);   % v1 in blue
\fill[black] (B) circle (1.5pt);   % v2
\fill[black] (C) circle (1.5pt);   % v3
\fill[black] (D) circle (1.5pt);   % v4
\fill[black] (E) circle (1.5pt);   % v5
\fill[black] (F) circle (1.5pt);   % v6
\fill[magenta] (G) circle (1.5pt);   % v7
\fill[black] (H) circle (1.5pt);   % v8

% --- Labels (manually placed) ---
\node at (-0.2, 0.3) {$v_1$};
\node at (1.2, 0.3) {$v_2$};
\node at (0.8, 0.65) {$v_3$};
\node at (1.5, 1.75) {$v_4$};
\node at (1.5, -1.05) {$v_5$};
\node at (1.8, 0.3) {$v_6$};
\node at (3.2, 0.3) {$v_7$};
\node at (2.2, 0.65) {$v_8$};

\end{tikzpicture}
\captionof{figure}{A $3$-admissible pinned graph.}
\end{minipage}

For the corresponding unpinned graph, it was shown in \cite{CIMP2021} that one has positive measure in $\R^{17}$ for the distance set induced by the double banana with the deleted edge, inside compact sets $E\subset \R^3$ with $\dim(E) > 3 - 1/9$. We note that our methods only provide the threshold
\[
\dim(E) > \frac{d+3}{2} = 3,
\]
so they do not yield any nontrivial result in $\R^3$.

\section{An algorithm - graph dismantling lemma}\label{sec: graphdismantling}

We want to figure out the minimal admissibility $k$ of a given pinned graph. When a graph has no pins, our algorithm recovers $k$ as the $k$-degeneracy of the graph. 

\begin{defn}[Pinned subgraph]
    We will say that a pinned graph $(H,\mathcal{P}_{H})$ is a \textbf{pinned subgraph} of the pinned graph $(G, \mathcal{P}_G)$ if $H$ is a pinned graph which is an induced subgraph of $G$, and $\mathcal{P}_H\subset \mathcal{P}_G$, i.e., the sets of pinned vertices in $H$ is a subset of the set of pinned vertices in $G$. 
\end{defn}

 We start with a simple lemma.

\begin{lem}\label{lem: pinnedsubgraphs}
    Let $(H,\mathcal{P}_H)$ be a pinned subgraph of $(G,\mathcal{P}_G)$. If $(G,\mathcal{P}_G)$ is $k$-admissible, then so is $(H,\mathcal{P}_H)$.
\end{lem}
\begin{proof}
Let $\mathcal{P}_G=\{v_1,v_2,\dots,v_m\}$ be the set of pins in $G$.
    There is a sequence of unpinned vertices $v_{m+1},\ldots,v_l$ in $\mathcal{V}_G$ such that $G$ can be constructed by starting with its pinned vertices and using legal $k$-admissible steps by adding the unpinned vertices in order. Remove any vertices from the list $v_1,v_2,\dots, v_l$ that are not in $\mathcal{V}_H$ and now construct $H$ using the remaining vertices in order; as the number of edges added at each step cannot go up, this proves $(H,\mathcal{P}_H)$ is also $k$-admissible.
\end{proof}

Now, we can describe a graph dismantling algorithm for checking if a (not necessarily singly) pinned graph is $k$-admissible. The algorithm is not unique, but we will show that the result of the algorithm does not depend on the choices made.\\

\textbf{$k$-ALGORITHM}: While there is at least one unpinned vertex of degree at most $k$, choose one such vertex and delete it (and all edges emanating from it). The procedure terminates once all unpinned vertices have been deleted or once all remaining unpinned vertices have degree strictly greater than $k$.

\begin{prop}\label{prop:algorithmkadm}
A pinned graph $G$ is $k$-admissible if and only if the $k$-algorithm deletes all unpinned vertices. In particular, the outcome of the $k$-algorithm is independent of the choices made when applying the $k$-algorithm.
\end{prop}

This proposition is very useful for determining the admissibility number of a pinned graph. Indeed, since the admissibility number is always greater than or equal to the minimum degree, the $k$-algorithm can be run repeatedly, starting from $k$ equaling the minimum degree, and incrementing until a value is reached that leads to the algorithm deleting all unpinned vertices. The minimal such value is then the admissibility number of the pinned graph.
\begin{proof}
    If the $k$-algorithm deletes all unpinned vertices, then we can construct the graph by adding in the deleted vertices in the opposite order; this implies that $G$ is $k$-admissible. Conversely, if $G$ is $k$-admissible, then we can choose to delete the vertices in the reverse order in which they were added by the allowed moves in the definition of $k$-admissibility. Thus, it suffices to show that the outcome of the $k$-algorithm does not depend on the choices made to justify the fact that we made choices in the previous argument. If $G$ is not $k$-admissible, then the $k$-algorithm can never delete all unpinned vertices; otherwise, the graph can be constructed legally with $k$-admissible moves by adding the vertices in the opposite order in which they were deleted. If $G$ is $k$-admissible, we first claim there is at least one unpinned vertex of degree at most $k$. Indeed, there exists an ordering of unpinned vertices $w_1,\ldots w_l$, and then the vertex $w_l$ must have degree at most $k$. Let $G'$ denote the graph obtained by deleting an unpinned vertex of degree at most $k$ from $G$. Then $G'$ is $k$-admissible by the earlier lemma. The claim then follows by induction (more precisely, we could fix the number of pins to be $m$ and induct on the number of unpinned vertices $j$; then we use the induction hypothesis to deal with $G'$).

\end{proof}

\begin{prop}
    For any unpinned graph, $k$-admissibility is equivalent to $k$-degeneracy.
\end{prop}
\begin{proof} This follows immediately from the algorithm described in \cite{Matula}.
\end{proof}

\begin{prop}\label{prop: adding pins}
    Let $(G,\mathcal{P})$ be a pinned graph with $|\mathcal{P}|=m_0+m$, and let $G_0=(G,\mathcal{P}')$ denote the same graph with pins in $\mathcal{P}'\subset \mathcal{P} $ and $|\mathcal{P}'|=m_0$, where $m_0,m\in\mathbb{Z}_{\geq 0}$. If $G_0$ is $k$-admissible, then $(G,\mathcal{P})$ is $(k+m)$-admissible.
\end{prop}
\begin{proof}
    Since $G_0$ is $k$-admissible, the definition implies there is an ordering on the vertices of $G $ $v_1,v_2,\ldots,v_{m_0}, v_{m_0+1}\dots v_l$ such that $\mathcal{P}'=\{v_1,\dots, v_{m_0}\}$ are the pins, and each vertex $v_i$ with $i>m_0$ has at most $k$ edges with all preceding vertices. Let $v_{i_1},\ldots,v_{i_m}$, $m_0< i_1<\dots <i_{m}\leq l$, denote the $m$ extra pins in $\mathcal{P}\backslash \mathcal{P}'$. Now, consider the following ordering for the unpinned vertices of $G$: $v_{m_0+1},v_{m_0+2},\ldots,\hat{v}_{i_1},\ldots,\hat{v}_{i_m},\ldots,v_l$ where the hats denote omitting that vertex. When we construct $G$ starting from the pins $v_1,\dots,v_{m_0},v_{i_1},\dots, v_{i_m}$ and adding the unpinned vertices using that ordering, each vertex $v_j$ has at most $k$ edges with the lower-indexed vertices and at most $m$ edges with $v_{i_1},\ldots,v_{i_m}$, giving the claim.
\end{proof}

We note that the proposition above is saturated by considering an $k$-star graph. If only one leaf is pinned, then the graph is $1$-admissible, and after adding $k-1$ pins in all the remaining leaves, one gets the expected $k$-admissibility for the pinned $k$-star graph.

\section{Special case: Multiple pinned chains and cycles}\label{sec: pinnednecklaces}

In this section, we present the details of how to prove the positive measure result for a pinned chain and a pinned cycle graph (with one or more pins). The threshold $(d+2)/2$ holds for such pinned configurations as a consequence of our structural Theorem \ref{thm: firststructuralthm}, but these specific graph structures are simple enough that one does not need the full strength of $L^2$ densities for $2$-stars as stated in Theorem \ref{thm:k starinIPPS}. It is enough to have positive measure for $2$-stars in the sense defined below, which is implied by the stronger $L^2$ statement.

\begin{defn}\label{def: fullset}
    Let $\nu$ be a finite measure compactly supported in $E\subset \R^n$. We will say that a Borel subset $F\subseteq E$ is $\nu$-full if $\nu(F)=\nu(E).$
\end{defn}

\begin{defn}\label{def:weak-form-k-star}
We say that $\alpha \in (0,d)$ is \emph{suitable for positive measure of $k$-stars} if the following holds: For every compact set $E \subset \mathbb{R}^d$ with $\dim(E) > \alpha$, every $s$-Frostman measure $\mu$ supported on $E$ with $s > \alpha$, and every collection of pairwise separated compact subsets 
$E_1, E_2, \dots, E_{k+1} \subset E$
satisfying $\mu(E_i) > 0$ for all $i=1,\dots,k+1$, one has that for 
\[
\mu_{E_1} \times \mu_{E_2} \times \cdots \times \mu_{E_k}
\text{-a.e. } (x_1,\dots,x_k) \in E_1 \times \cdots \times E_k,
\]
it holds that
\[
\mathcal{L}^k\!\left(\Delta^{k\text{-star}}_{x_1,\dots,x_k}(E_{k+1})\right) > 0.
\]
In other words, pinned positive measure of $k$-stars holds for $k$-tuples of pins $(x_1,x_2,\dots ,x_k)$ contained in a $(\Pi_{i=1}^k \mu_{E_i})$-full subset of $\Pi_{i=1}^{k}E_i\subset \R^{kd}$.
\end{defn}

\begin{rem}
    If $\alpha\in(0,d)$ is suitable for positive measure of $k$-stars, then it is also suitable for positive measure of $k'$ stars, for any $1\leq k'<k$. 
\end{rem}

\begin{lem}\label{lem: connected components}
Let $G$ be a graph with connected components $G_1,\cdots, G_\ell$. Let $\Phi$ be an edge function vector of $G$ and $\Phi^i$ be its restriction on the connected component $G_i$. For any pin set $\mathcal{P}$ of $G$, let $\mathcal{P}_i$ denote its restriction on the component $G_i$. Let $E\subset \mathbb{R}^d$ be a compact set such that for all $i=1,\cdots, \ell$, its pinned $(G_i,\Phi^i)$-graph set $\Delta^{G_i,\Phi^i}_{\mathcal{P}_i,\{x^i_j\}_j}$ has positive Lebesgue measure for some points $\{x^i_j\}_j$. Then, its $(G,\Phi)$-graph set $\Delta^{G,\Phi}_{\mathcal{P},\{x^i_j\}_{i,j}}$ must also have positive Lebesgue measure.
\end{lem}

\begin{proof}
Observe that $\Delta^{G,\Phi}_{\mathcal{P},\{x^i_j\}_{i,j}}$ is precisely the Cartesian product of all the $\Delta^{G_i,\Phi^i}_{\mathcal{P}_i,\{x^i_j\}_j}$, $i=1,\cdots, \ell$. The claimed result follows immediately.
\end{proof}

\begin{thm}\label{thm: pinnedchains}
    Assume $\alpha\in (0,d)$ is suitable for positive measure of $2$-stars. Let $l\geq 2$ and let $Ch$ be an $(l-1)$-chain with $m\geq 1$ pinned vertices listed in $\mathcal{P}$,  such that no pair of vertices in $\mathcal{P}$ shares an edge. Then, if $E\subset \R^d$ is a compact set satisfying that $\dim(E)>\alpha$, $\mu$ Frostman measure in $E$ and $E_1,E_2,\dots,E_l\subset E$ are separated compact subsets with positive $\mu$ measure, one has
$$\mathcal{L}^{l-1}\left(\Delta_{\mathcal{P},x_1,x_2,\dots ,x_m}^{Ch}(E_{m+1},\dots ,E_l)\right)>0$$
for $\Pi_{i=1}^m\mu_{E_i}$ a.e. $(x_1,x_2,\dots ,x_m)\in E_1\times \dots \times E_m$.
\end{thm}

\begin{proof}
    Let us prove it by induction on $n$, the number of unpinned vertices in the chain (then $n+m=l$). 
    If $n=1$, that means all but one vertex are pinned. This is either a $1$-chain with $1$ pin or a $2$-chain with two pins at its endpoints. Positive measure holds for $\mu_{E_1}$ a.e. point $x_1$ or $\mu_{E_1}\times \mu_{E_2}$  a.e. $(x_1,x_2)$, respectively, as a consequence of the assumption on the threshold $\alpha$.

    Next, assume that $n\geq 2$ and that the conclusion is known for chains which have at most $n-1$ unpinned vertices. Take any $2$-admissible construction order $$\mathcal{O}=(v_1,v_2,\dots ,v_m,v_{m+1},v_{m+2},\dots ,v_{m+n}=v_l)$$ for $Ch$. Let us delete $v_{m+n}$ and its adjacent edges, which gives us a subgraph $\tilde{Ch}$ (potentially disconnected). There are two cases: either $v_{m+n}$ is a leaf (only one edge emanating from $v_{m+n}$) or $v_{m+n}$ is not a leaf.

    \textbf{Case 1: $v_{m+n}$ is a leaf}. In this case, after deleting $v_{m+n}$ and its adjacent edge, we are left with a new chain $\tilde{Ch}$ with pins in the same vertices as before, but now one fewer unpinned vertex. Let $1\leq r\leq m+n-1$ be the index of the vertex connected to $v_{m+n}$ by an edge in $Ch$.
    
    \textbf{Subcase 1.1}: $v_r$ is unpinned vertex:  Since for $\mu_{E_r}$ a.e. $x_r$ we have positive measure for $\Delta_{x_r}(E_{m+n})$, we can prune $E_r$ to a smaller compact set $E_r'\subset E_r$ such that any pin $x_r\in E_r'$ gives positive measure for $\Delta_{x_r}(E_{m+n})$. Then apply the induction hypothesis to $\tilde{Ch}$ and the sets $\{E_i\}_{1\leq i\neq r\leq m+n-1}$ and $E_r'$ to get that for $\Pi_{i=1}^m\mu_{E_i}$ a.e $(x_1,x_2,\dots,x_m)\in \Pi_{i=1}^{m} E_i$  one has $\Delta^{\tilde{Ch}}_{\mathcal{P},x_1,\dots ,x_m}(E_{m+1},\dots E_{r-1},E_{r}',E_{r+1},\dots, E_{m+n-1})$. From Fubini, for each such $m$-tuple $(x_1,x_2,\dots, x_m)$ the pinned distance set
    $$\Delta^{Ch}_{\mathcal{P},x_1,\dots ,x_m}(E_{m+1},\dots E_{r-1},E_{r}',E_{r+1},\dots, E_{m+n})$$ has positive measure. Since $E_r'\subset E_r$, that is enough for the result to follow.

    \textbf{Subcase 1.2}: $v_r$ is a pinned vertex. By the induction hypothesis, one has that for $\mu_{E_r}$ a.e $x_r$ in $E_r$, $\Delta_{x_r}(E_{m+n})$ has positive Lebesgue measure, and for $$\mu_{E_r}\times \Pi_{1\leq i\neq r\leq m}\mu_{E_i} \text{ a.e. } (x_r,x_1, \dots, x_{r-1}, x_{r+1}, \dots ,x_m )\in E_r\times \Pi_{1\leq i\neq r\leq m} E_i$$ one has $\Delta^{\tilde{Ch}}_{\mathcal{P},x_1,\dots, x_m}(E_{m+1}, \dots ,E_{m+n-1})$ has positive measure. That is enough to guarantee that there is a $(\Pi_{i=1}^{m}\mu_{E_i})$-full subset of $E_1\times \dots \times  E_m$ such that for any $(x_1, \dots ,x_m)$ in such set, both things happen. Indeed, say $\mathcal{E}^{(m)}\subset \Pi_{i=1}^m E_i$ is the $(\Pi_{i=1}^{m}\mu_{E_i})$-full set and $\mathcal{E}^{(1)}_{r}\subset E_r$ is $\mu_{E_r}$-full set we had in the previous steps, then take $$\mathcal{F}^{(m)}=\mathcal{E}^{(m)}\cap ( E_1\times \dots \times E_{r-1}\times \mathcal{E}_r^{(1)}\times E_{r+1}\times \dots E_m).$$
    
    Any $m$-tuple $(x_1,\dots ,x_m)\in \mathcal{F}^{(m)}$ gives positive measure for  $$\Delta^{Ch}_{\mathcal{P},x_1,\dots ,x_m}(E_{m+1}, E_{m+2} \dots ,E_{m+n})=\Delta^{\tilde{Ch}}_{\mathcal{P},x_1,\dots ,x_m}(E_{m+1}, E_{m+2} \dots ,E_{m+n-1})\times \Delta_{x_r}(E_{m+n}).$$ 

    \textbf{Case 2}: $v_{m+n}$ is not a leaf.

    \textbf{Subcase 2.1.} $v_{m+n}$ is connected to two pinned vertices in $Ch$.
    When deleting $v_{m+n}$ and the adjacent vertices, we are left with a disconnected graph $\tilde{Ch}$  made of two disconnected subchains. We apply the induction hypothesis to each of these subchains separately. Using Lemma \ref{lem: connected components}, this produces full product measure of $m$-tuples of pins $(x_1,x_2,\dots ,x_m)\in E_1\times \dots \times E_m$ such that $\Delta^{\tilde{Ch}}_{\mathcal{P},x_1,\dots, x_m}(E_{m+1},E_{m+2},\dots ,E_{m+n-1})$ has positive Lebesgue measure. If $r_1\neq r_2\in \{1,2,\dots ,m\}$ are the indices of the vertices connected to $v_{m+n}$ in $Ch$, we also know that $\mu_{E_{r_1}}\times \mu_{E_{r_2}}$ a.e. $(x_{r_1},x_{r_2})\in E_{r_1}\times E_{r_2}$ gives $\mathcal{L}^2(\Delta^{2-star}_{x_{r_1},x_{r_2}}(E_{m+n}))>0$.  By intersecting a couple of $(\Pi_{i=1}^{m}\mu_{E_i})$-full sets, there is still a $(\Pi_{i=1}^{m}\mu_{E_i})$-full set of $m$-tuples $(x_1,x_2,\dots ,x_m)$ for which both pinned distance sets have positive measure. Any such $m$-tuple will produce positive measure for $$\Delta_{\mathcal{P},x_1,\dots, x_m}^{Ch}(E_{m+1},\dots ,E_{m+n})=\Delta_{\mathcal{P},x_1,\dots,x_m}^{\tilde{Ch}}(E_{m+1},\dots ,E_{m+n-1})\times \Delta_{x_{r_1},x_{r_2}}^{2-star}(E_{m+n}).$$
      
      \textbf{Subcase 2.2.} $v_{m+n}$ is connected to one pinned vertex and one unpinned vertex in $Ch$. Let $p$ and $u$ be the indices of the pinned and unpinned vertices that $v_{m+n}$ is connected to in $Ch$, respectively. When deleting $v_{m+n}$ and adjacent edges we are left with two subchains, $Ch_u,Ch_p$, with endpoints in $v_u$ and $v_p$, respectively.  Since $\alpha$ is suitable for positive measure of two stars, we have that for $\mu_{E_p}\times \mu_{E_u}$ a.e. pair $(x_p,x_u)\in E_p\times E_u$ one has $\mathcal{L}^2(\Delta^{2-star}_{x_p,x_u}(E_{m+n}))>0$. In particular, there is a set $E_{p}'\subset E_p$ of full $\mu_{E_p}$ measure such that for every $x_p\in E_p'$, 
      $$E_u'(x_p):=\{x_u\in E_u\colon \mathcal{L}^2(\Delta_{x_p,x_u}^{2-star}(E_{m+n}))>0\} \text{ has full measure }\mu_{E_u}(E_u).$$

      Let $\mathcal{P}_p$ and $\mathcal{P}_u$ be the set of pins induced in $Ch_p$ and $Ch_u$ from $(Ch,\mathcal{P})$, respectively. Let $m_p=|\mathcal{P}_p|$ and $m_u=|\mathcal{P}_u|$, so that $m_p+m_u=m$.
      Apply the induction hypothesis in $(Ch_p,\mathcal{P}_p)$ with $E_p$ replaced by $E_p'$, getting a full product measure of $m_p$-tuples of good pins.
    Next, for each $x_p\in E_p'$ one can take a compact set $K_u(x_p)\subset E_u'(x_p)$ with $\mu_{E_u}(K_u(x_p))>0$ and apply the induction hypothesis in the subchain of endpoint $v_u$ with $K_u(x_p)$ playing the role of $E_u$, again one get subset of full product measure for $m_u$ tuples of good pins for $(Ch_u,\mathcal{P}_u)$. By Fubini's theorem, this implies full measure for the entire pinned distance set of $Ch$.

    \textbf{Subcase 2.3.} $v_{m+n}$ is connected by edges to two unpinned vertices $v_{r_1}$ and $v_{r_2}$. In this case, we can prove a stronger result, where we pin at the extra location  $v_{m+n}$. For that, apply the induction hypothesis to the two subchains that one gets by splitting $Ch$ at $v_{m+n}$ where the pins in each subchain are two ones inherited from $(Ch,\mathcal{P})$ plus a new pin at $v_{m+n}$. In each subchain, one does have less than $n$ unpinned vertices, and by the induction hypothesis, we get a full product measure of good pins. By intersecting a couple of full sets, we get the desired result. 
\end{proof}

Next, we prove a positive measure result for pinned $l$-cycles, $l\geq 4$. The basic idea is to view the necklace as a pair of chains that share their endpoints; such a perspective was used in \cite{GIP17} as well, though in that paper, the Cauchy-Schwarz method utilized required the additional restriction that the cycle was of even length and constant gap. We believe that our theorem is one of the first results in the literature that holds for cycles of arbitrary lengths in the continuous setting. Moreover, unlike the method in \cite{GIP17}, which implicitly splits the necklace into a pair of symmetric chains in an $L^2$ based estimate, our approach uses the input of the positive measure of the chain sets as a black box and does not need the two chains to be of the same length.

%\textcolor{red}{Also, another difference is that they don't use the Lebesgue measure result for the subchains as a black box, right?}

\begin{thm}
Assume that $\alpha\in (0,d)$ is suitable for $2$-stars. Let $l\geq 4$. Let $C_l$ be an $l$-cycle graph with $m\geq0$ pins listed in $\mathcal{P}$, such that no pair of vertices in $\mathcal{P}$ shares an edge. Then, if $E\subset \R^d$ is a compact set satisfying that $\dim(E)>\alpha$, $\mu$ Frostman measure in $E$ and $E_1,E_2,\dots,E_l\subset E$ are separated compact subsets with positive $\mu$ measure, one has positive measure of pinned at $\mathcal{P}$ $l$-cycles in the sense that
$$\mathcal{L}^l\left(\Delta_{\mathcal{P},x_1,x_2,\dots ,x_m}^{C_l}(E_{m+1},\dots,E_l)\right)>0$$
for $\Pi_{i=1}^m\mu_{E_i}$ a.e. $(x_1,x_2,\dots ,x_m)\in E_1\times \dots \times E_m$. 
\end{thm}

\begin{rem}
    We remark that the case $l=3$ (triangles) is not included in the theorem above, and for that case, we still need the stronger $L^2$ statement for $2$-stars.
\end{rem}

 \begin{proof}
     We can assume without loss of generality that $m\geq 2$. If $m=1$ or $m=0$, we can pick one or two extra vertices to pin, respectively, in a way that we get a cycle with two nonconsecutive pins (for this, we must assume $l\geq 4$ so that pinning two non-consecutive vertices in the cycle is possible). The fact that one has a positive measure for this doubly pinned cycle (case $m=2$) implies the result for a subset of those pins.

     In the case $m\geq 2$, label the pinned vertices $v_1,v_2, \dots ,v_m$. The pinned cycle can be split into two pinned chains $Ch_1, Ch_2$ with pinned endpoints $v_1,v_2$ (and possibly some other pins along $Ch_i$).

     \begin{figure}[h]
\centering

     \begin{tikzpicture}[scale=2, line cap=round, line join=round]

% -------- Styles --------
\tikzset{
  blk/.style={circle, fill=black, inner sep=1.4pt},
  grn/.style={circle, fill=magenta, inner sep=1.6pt},
  ed/.style={thick},
  gap/.style={gray, dashed, thick, -}
}

% -------- Coordinates (placed around an "almost circle") --------
% Top gap endpoints (magenta)
\coordinate (GtopL) at (-0.55,  1.05);
\coordinate (GtopR) at ( 0.55,  1.05);

% Bottom gap endpoints (magenta)
\coordinate (GbotL) at (-0.40, -1.05);
\coordinate (GbotR) at ( 0.40, -1.05);

% Left chain (4 edges = 5 vertices): GtopL -> B1 -> B2 -> B3 -> GbotL
\coordinate (B1) at (-1.10,  0.70);
\coordinate (B2) at (-1.25,  0.00);
\coordinate (B3) at (-0.95, -0.65);

% Right chain (5 edges = 6 vertices): GbotR -> C1 -> C2 -> C3 -> C4 -> GtopR
\coordinate (C1) at ( 0.95, -0.75);
\coordinate (C2) at ( 1.25, -0.10);
\coordinate (C3) at ( 1.10,  0.50);
\coordinate (C4) at ( 0.85,  0.85);

% -------- Draw the two disconnected chains --------
\draw[ed] (GtopL) -- (B1) -- (B2) -- (B3) -- (GbotL);
\draw[ed] (GbotR) -- (C1) -- (C2) -- (C3) -- (C4) -- (GtopR);

% -------- Draw the dashed "missing edges" --------
\draw[gap] (GtopL) -- (GtopR);
\draw[gap] (GbotL) -- (GbotR);

% -------- Vertices --------
\node[grn] at (GtopL) {};
\node[grn] at (GtopR) {};
\node[grn] at (GbotL) {};
\node[grn] at (GbotR) {};
\node[grn] at (C3) {};
\node[grn] at (B2) {};

\foreach \P in {B1,B3,C1,C2,C4}{
  \node[blk] at (\P) {};
}

% -------- Labels (edit these to match your notation) --------
\node[magenta, above] at (GtopL) {$v_1$};
\node[magenta, above] at (GtopR) {$v_1$};

\node[magenta, below] at (GbotL) {$v_2$};
\node[magenta, below] at (GbotR) {$v_2$};

\node[left]  at (B1) {$v_{u_{1,1}}$};
\node[magenta,left]  at (B2) {$v_{p_{1,1}}$};
\node[below] at (B3) {$v_{u_{1,2}}$};

\node[below] at (C1) {$v_{u_{2,3}}$};
\node[right] at (C2) {$v_{u_{2,2}}$};
\node[magenta,right] at (C3) {$v_{p_{2,1}}$};
\node[right] at (C4) {$v_{u_{2,1}}$}; % placeholder; change as needed

\node at (-2.25, 0) {$Ch_1$};
\node at (2.25, 0) {$Ch_2$};
\end{tikzpicture}
\caption{A split of a pinned $9$-cycle (with $4$ pins) into two pinned chains sharing pins at $v_1,v_2$. 
The dashed gray edges indicate how to glue the two chains back to the original pinned cycle. The vertices in magenta are the pinned ones.}
\label{fig:twochainssplit}
\end{figure}
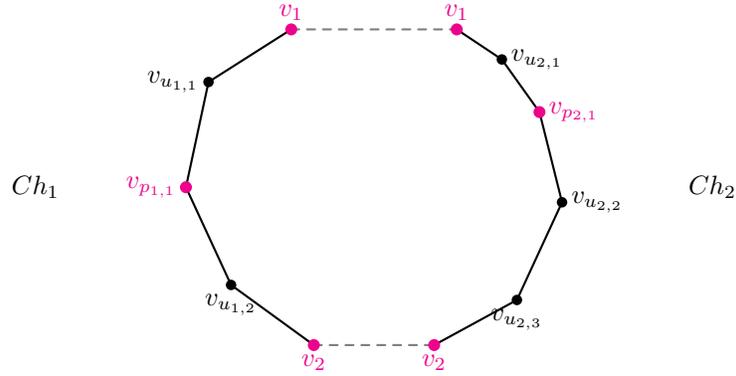

     For $i=1,2$, let $M_i$ and $N_i$ be the number of pinned and unpinned vertices in $Ch_i$, respectively. Note that $M_1+M_2=m+2$ because the chains share two of their pins. For each $i=1,2$ say $\mathcal{P}_{i}=\{v_1,v_2\}\cup \{v_{p_{i,j}}\}_{j=1}^{M_i}$ are the pinned vertices of $Ch_i$ and $\{v_{u_{i,j}}\}_{j=1}^{N_i}$ are the unpinned vertices of $Ch_i$, see Figure \ref{fig:twochainssplit} for an example. For each $i=1,2$, applying Theorem \ref{thm: pinnedchains} to the pinned chain graphs $(Ch_i,\mathcal{P}_i)$ we get that for $\mu_{E_1} \times\mu_{E_2}\times \Pi_{j=1}^{M_i} \mu_{E_{p_{i,j}}}$ a.e. $(x_1,x_2,x_{p_{i,1}},\dots, x_{p_{i,M_i}})$ one has 
     
     \begin{equation}\label{eq:positemeasurepieceofthecycle}
     \mathcal{L}^{|\mathcal{E}_{i}|}\left(\Delta^{Ch_i}_{\mathcal{P}_i,x_1,x_2, x_{p_{i,1}},\dots, x_{p_{i,M_i}}}(E_{u_{i,1}},E_{u_{i,2}},\dots , E_{u_{i,N_i}})\right)>0,
     \end{equation}
     where $|\mathcal{E}_i|$ is the number of edges in $Ch_i$.

For each $i=1,2$, there exists $\mathcal{E}_{1,2}^i\subset E_1\times E_2$ of full $\mu_{E_1}\times \mu_{E_2}$ measure, such that, for each $(x_1,x_2)\in \mathcal{E}_{12}^i$, $$\Pi_{j=1}^{M_i} \mu_{E_{p_{i,j}}}(\{(x_{p_{i,1}},x_{p_{i,2}},\dots, x_{p_{i,M_i}})\colon (\ref{eq:positemeasurepieceofthecycle}) \text{ holds }\})=\Pi_{j=1}^{M_i} \mu(E_{p_{i,j}}).$$

      Then the set $\mathcal{E}_{12}:=\mathcal{E}_{12}^1\cap \mathcal{E}_{12}^2\subset E_1\times E_2$ still has full $\mu_{E_1}\times \mu_{E_2}$ measure. For each $(x_1,x_2)\in \mathcal{E}_{12}$, 
      $$\mathcal{L}^l(\Delta^{C_l}_{\mathcal{P},x_1,\dots, x_m}(E_{m+1},\dots ,E_l))=\Pi_{i=1}^2 \mathcal{L}^{|\mathcal{E}_{i}|}\left(\Delta^{Ch_i}_{\mathcal{P}_i,x_1,x_2, x_{p_{i,1}},\dots, x_{p_{i,M_i}}}(E_{u_{i,1}},E_{u_{i,2}},\dots , E_{u_{i,N_i}})\right)>0, $$
      for all $(x_{p_{1,1}},\dots ,x_{p_{1,M_1}},x_{p_{2,1}},\dots ,x_{p_{2,M_2}})$ in a set of full $\Pi_{i=1}^2\Pi_{j=1}^{M_i} \mu_{E_{p_{i,j}}}$ measure.
 \end{proof}

\section{Proof of general structural Theorems \ref{thm: firststructuralthm} and \ref{thm: secondstructuralthm}}\label{sec: proofofstructuraltheorems}

Given a pinned graph $(G,\mathcal{P})$ with pins in $\mathcal{P}=(v_1,v_2,\dots ,v_m)$ and construction order $\mathcal{O}=(v_1,v_2, \dots ,v_l)$, we let $\eta_{j,\mathcal{O}}$ denote the index set of vertices sharing an edge with $v_j$ that come before $v_j$ in the ordering $\mathcal{O}$; we will often suppress the dependence on $\mathcal{O}$ from the notation. Denote $\epsilon_j=|\eta_j|$, that is, $\epsilon_j$ is the the back degree of $v_j$ with respect to the ordering $\mathcal{O}$. We also let $\eta_{j,P}$ denote the subset of $\eta_j$ consisting of indices of pinned vertices and $\eta_{j,U}$ be the subset consisting of indices of unpinned vertices; analogously, use $\epsilon_{j,P}$ and $\epsilon_{j,U}$ to denote their cardinalities. Let $l$ be the total number of vertices in the graph $G$, recall that $m$ denotes the number of pinned vertices, and let $n:=l-m$ (number of unpinned vertices). For ease of notation, denote $x=(x_1,\cdots, x_m)$, and for $r\ge m+1$, we let $y_{(r)}$ denote the $(r-m)$-tuple $(y_{m+1},\ldots,y_{r})$. Any mention of a full subset in what follows refers to the notion introduced earlier in Definition \ref{def: fullset}.

\begin{defn}\label{def: goodpin}
    Suppose $m\geq 1$ (at least one pin). For a compact set $E\subset \R^d$ supporting a Frostman measure $\mu_E$, fix separated compact subsets $E_1,\ldots, E_l\subset E$ with $\mu(E_j)>0$ for all $j=1,2\dots, l$, and let $\mu_{E_j}$ denote the restriction of $\mu$ to $E_j$. We say that an $m$-tuple of pins $(x_1,\ldots,x_m)\in E_1\times E_2\times \dots\times  E_m$ is good for $(G,\mathcal{P},\mathcal{O},\{E_i\}_{i=1}^l)$, if the following holds true:

\begin{enumerate}\item $\left(d_{(x_i)_{i\in\eta_{m+1,P}}}^{\epsilon_{m+1}-star}\right)_*(\mu_{E_{m+1}})\in L^2(\R^{\epsilon_{m+1}})$.
\item There exists $(E_{m+1})^G_{x}\subset E_{m+1}$ such that $(E_{m+1})^G_{x}$ is $\mu_{m+1}$-full, and that for each $y_{m+1}\in (E_{m+1})^G_{x}$, the pushforward measure $\left(d_{(x_i)_{i\in\eta_{m+2,P}},(y_{j})_{j\in \eta_{m+2,U}}}^{\epsilon_{m+2}-star}\right)_*(\mu_{E_{m+2}})\in L^2(\R^{\epsilon_{m+2}})$. 
\item For each $y_{m+1}\in (E_{m+1})^G_{x}$, the set identified in the above, there exists $(E_{m+2})^G_{x,y_{(m+1)}}\subset E_{m+2}$ such that $(E_{m+2})^G_{x, y_{(m+1)}}$ is $\mu_{m+2}$-full, and that for each $y_{m+2}\in (E_{m+1})^G_{x, y_{(m+1)}}$, the pushforward measure $\left(d_{(x_i)_{i\in\eta_{m+3,P}},(y_j)_{j\in \eta_{m+3, U}}}^{\epsilon_{m+3}-star}\right)_*(\mu_{E_{m+3}})\in L^2(\R^{\epsilon_{m+3}})$.
\item $\cdots$
\item For each $y_{m+1}\in (E_{m+1})^G_{x}$, each $y_{m+2}\in (E_{m+2})^G_{x, y_{(m+1)}}$, $\cdots$, and each $y_{l-2}\in (E_{l-2})^G_{x,y_{(l-3)}}$, chosen from the previously identified sets, there exists $(E_{l-1})^G_{x,y_{(l-2)}}\subset E_{l-1}$ such that $(E_{l-1})^G_{x, y_{(l-2)}}$ is $\mu_{l-1}$-full, and that for each $y_{l-1}\in (E_{l-1})^G_{x, y_{(l-2)}}$, the pushforward measure $\left(d_{(x_i)_{i\in\eta_{l,P}},(y_j)_{j\in \eta_{l, U}}}^{\epsilon_{l}-star}\right)_*(\mu_{E_{l}})\in L^2(\R^{\epsilon_l})$.
\end{enumerate}

\end{defn}

By the $\ldots$ step, we refer to a series of conditions indexed by $\Upsilon$ (where $\Upsilon$ varies over values so that $\Upsilon\ge 1$ and $l-\Upsilon-3\ge m+1$) and of the form: \\
``For each $y_{m+1}\in (E_{m+1})^G_{x}$, each $y_{m+2}\in (E_{m+2})^G_{x, y_{(m+1)}}$, $\cdots$, and each $y_{l-\Upsilon-2}\in (E_{l-\Upsilon-2})^G_{x,y_{(l-\Upsilon-3)}}$, chosen from the previously identified sets, there exists $(E_{l-\Upsilon-1})^G_{x,y_{(l-\Upsilon-2)}}\subset E_{l-\Upsilon-1}$ such that $(E_{l-\Upsilon-1})^G_{x, y_{(l-\Upsilon-2)}}$ is $\mu_{l-\Upsilon-1}$-full, and that for each $y_{l-\Upsilon-1}\in (E_{l-\Upsilon-1})^G_{x, y_{(l-\Upsilon-2)}}$, the pushforward measure $\left(d_{(x_i)_{i\in\eta_{l-\Upsilon,P}},(y_j)_{j\in \eta_{l-\Upsilon, U}}}^{\epsilon_{l-\Upsilon}-star}\right)_*(\mu_{E_{l-\Upsilon}})\in L^2(\R^{\epsilon_{l-\Upsilon}})$." \\
That is, it covers the natural intermediaries between conditions (3) and (5) in the definition, where we think of decrementing $\Upsilon$ each time until the final condition (5) is the case where $\Upsilon=0$.

\begin{ex}
    
Consider the toy graph $G$ illustrated in Figure \ref{fig:toypinnedforinduction} where $\mathcal{P}=(v_1,v_2)$ are the pinned vertices and with the $4$-admissible ordering $\mathcal{O}=(v_1,v_2,\dots,v_6)$. Suppose that separated compact sets $E$ and $E_1,E_2, \dots ,E_6$ as before have been fixed.

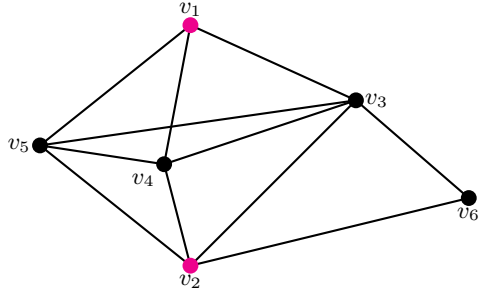
\begin{figure}[h]
\centering
\begin{tikzpicture}[scale=1.0, >=latex]

% ---- styles ----
\tikzset{
  edge/.style    = {thick},
  pin/.style     = {circle, fill=magenta, inner sep=0pt, minimum size=6pt},
  nodept/.style  = {circle, fill=black,   inner sep=0pt, minimum size=6pt},
  lbl/.style     = {font=\small}
}

% ---- coordinates ----
\coordinate (v5) at (-3, 0);
\coordinate (v1) at (-1, 1.6);
\coordinate (v2) at (-1,-1.6);
\coordinate (v4) at (-1.35,-0.25);
\coordinate (v3) at ( 1.2, 0.6);
\coordinate (v6) at ( 2.7,-0.7);
%\coordinate (v7) at ( 4.8,-0.7);

% ---- edges (all straight) ----
% left cluster
\draw[edge] (v5) -- (v1);
\draw[edge] (v5) -- (v4);
\draw[edge] (v5) -- (v3);
\draw[edge] (v5) -- (v2);

\draw[edge] (v1) -- (v4);
\draw[edge] (v1) -- (v3);

\draw[edge] (v2) -- (v4);
\draw[edge] (v2) -- (v3);

\draw[edge] (v4) -- (v3);

% right side
\draw[edge] (v3) -- (v6);
\draw[edge] (v2) -- (v6);
%\draw[edge] (v6) -- (v7);

% ---- vertices ----
\node[pin]    at (v1) {};
\node[pin]    at (v2) {};
\node[nodept] at (v3) {};
\node[nodept] at (v4) {};
\node[nodept] at (v5) {};
\node[nodept] at (v6) {};
%\node[nodept] at (v7) {};

% ---- labels ----
\node[lbl, above]      at (v1) {$v_1$};
\node[lbl, below]      at (v2) {$v_2$};
\node[lbl, right]      at (v3) {$v_3$};
\node[lbl, below left] at (v4) {$v_4$};
\node[lbl, left]       at (v5) {$v_5$};
\node[lbl, below]      at (v6) {$v_6$};
%\node[lbl, right]      at (v7) {$v_7$};

\end{tikzpicture}
\caption{Toy pinned graph.}
  \label{fig:toypinnedforinduction}
\end{figure}

A pair $(x_1,x_2)\in E_1\times E_2$ is $(G,\mathcal{P},\mathcal{O},\{E_i\}_{i=1}^6)$ good if:
\begin{enumerate}
    \item $(d^{2-star}_{x_1,x_2})_{*}(\mu_{E_3})\in L^2(\R^2)$;
    \item There exists a $\mu_{E_3}$-full subset $(E_3)_{x_1,x_2}^{G}\subset E_3$ such that for all $y_3\in (E_3)_{x_1,x_2}^{G}$ 
    $$(d^{3-star}_{x_1,x_2,y_3})_{*}(\mu_{E_4})\in L^{2}(\R^3);$$
    \item For each $y_3\in (E_3)_{x_1,x_2}^{G} $, there exists a $\mu_{E_4}$-full measure subset $(E_4)_{x_1,x_2,y_3}^{G}\subset E_4$  such that for all $y_4\in (E_4)^G_{x_1,x_2,y_3} $
    $$(d^{4-star}_{x_1,x_2,y_3,y_4})_{*}(\mu_{E_5})\in L^{2}(\R^4) ;$$
    \item For each $y_3\in (E_3)_{x_1,x_2}^{G} $ and $y_4\in (E_4)_{x_1,x_2,y_3}^G$ there exists a $\mu_{E_5}$-full measure subset $(E_5)^G_{x_1,x_2,y_3,y_4}= E_5$ such that for all $y_5\in E_5$
    $$(d^{2-star}_{x_2,y_3})_{*}(\mu_{E_6})\in L^2(\R^2).$$
\end{enumerate}

With this definition of good pins in the example above, we can produce a positive measure of distance graphs via a Fubini-like argument. Indeed, we lower bound the Lebesgue measure of the distance set we care about by an iterated integral

\begin{equation*}
    \begin{split}
        \mathcal{L}^{11}(\Delta^G_{x_1,x_2}(E_3,E_4,\dots ,E_6))&\\
\geq{\int_{\Delta_{x_1,x_2}^{2-star}((E_3)^G_{x_1,x_2})}} &{(\int_{\Delta_{x_1,x_2,y_3}^{3-star}((E_4)^G_{x_1,x_2,y_3})}} {(\int_{\Delta_{x_1,x_2,y_3,y_4}^{4-star}((E_5)^G_{x_1,x_2,y_3,y_4})}} \\
        &{(\int_{\Delta_{x_2,y_3}^{2-star}(E_6)}dt_{10}dt_{11})}
{dt_6dt_7dt_8dt_9)}{dt_3dt_4dt_5)}{dt_1dt_2}
    \end{split}
\end{equation*}
where for each $(t_1,t_2)\in \Delta_{x_1,x_2}^{2-star}((E_3)^G_{x_1,x_2})$ we pick $y_3=y_3(t_1,t_2)\in (E_3)^G_{x_1,x_2} $ such that $|y_3-x_1|=t_1$ and $|y_3-x_2|=t_2$, then for each $(t_3,t_4,t_5)\in \Delta_{x_1,x_2,y_3}^{3-star}((E_4)^G_{x_1,x_2,y_3})$ we pick some $y_4=y_4(t_3,t_4,t_5)\in (E_4)^G_{x_1,x_2,y_3}$ such that $(t_3,t_4,t_5)=d^{3-star}_{x_1,x_2,y_3}(y_4)$.%, then for each $(t_6,t_7,t_8,t_9)\in \Delta_{x_1,x_2,y_3,y_4}^{4-star}((E_5)^G_{x_1,x_2,y_3,y_4})$ we pick some $y_5=y_5(t_6,t_7,t_8,t_9)\in (E_5)^G_{x_1,x_2,y_3,y_4}$ such that $(t_6,t_7,t_8,t_9)=d^{4-star}_{x_1,x_2,y_3,y_4}(y_5)$.

From the $L^2$ property for the pushforward measures in the definition of $(x_1,x_2)$ good pin, combined with the fact that for each $3\leq i\leq 5$, $\mu_{E_i}((E_i)^G_{x,y(i-1)})>0$ we can check that the distance sets showing up in the $4$-fold integral in the display above have positive Lebesgue measure. Indeed take a compact set $(K_r)^G_{x,y(r-1)}\subset (E_r)^G_{x,y(r-1)}$ with $\mu_{E_r}((K_r)^G_{x,y(r-1)})>0$, then Lemma \ref{Lemma: L2lemma} implies that $\Delta^{\epsilon_r-star}_{(x_i)_{i\in \eta_{r,P}},(y_j)_{j\in \eta_{r,U}}}((K_r)^G_{x,y(r-1)})>0$. Therefore $\mathcal{L}^{11}(\Delta^G_{x_1,x_2}(E_3,E_4,\dots ,E_6))>0.$

\end{ex}
\medskip

The next theorem is the main technical result needed to conclude the structural Theorem \ref{thm: firststructuralthm}.

\begin{thm}\label{thm:inductionstepinstructuralthm1}
    Suppose $(G,\mathcal{P})$ is a $k$-admissible pinned graph with $m\geq 1$ pinned vertices $\mathcal{P}=(v_1,v_2,\dots ,v_m)$ and any fixed $k$-admissible ordering $\mathcal{O}=(v_1,v_2,\dots ,v_l)$. Then, for any compact set $E\subset \R^d$ with $\dim(E)>\frac{d+k}{2}$ and any $s$-Frostman measure $\mu$ supported in $E$ with $s>\frac{d+k}{2}$ and any separated compact subsets $E_i$ of $E$ with $\mu(E_i)>0$, where $1\leq i\leq l$, one has that  $\mu_{E_1}\times \mu_{E_2}\times\dots \times  \mu_{E_m}$ almost every $m$-tuple $(x_1,x_2,\dots ,x_m)\in E_1\times E_2\times \dots E_m$ is $(G,\mathcal{P},\mathcal{O},\{E_i\}_{i=1}^l)$ good.
\end{thm}

We postpone the proof of Theorem \ref{thm:inductionstepinstructuralthm1} to the end of the section and now build our way to our structural theorem for Euclidean distances. First, we need a simple lemma about pushforward measures in $L^p$.

\begin{lem}\label{Lemma: L2lemma}
    Let $\mu_F$ be a Frostman measure compactly supported in $F$. Let $\Phi^k:F\rightarrow \R^k$ be a continuous map such that the pushforward measure $(\Phi^k)_{*}(\mu_F)$ is absolutely continuous with respect to Lebesgue measure in $\R^k$ and $\frac{d(\Phi^k)_{*}(\mu_F)}{dx}\in L^p(\R^k)$, for some $1\leq p\leq \infty$, which we will simply denoted by $(\Phi^k)_{*}(\mu_F)\in L^p$. Then for any subset $F'\subset F$ with $\mu_F(F')>0$ and $\mu_{F'}$ the restriction of $\mu_F$ to $F'$, one also has $(\Phi^k)_{*}(\mu_{F'})\in L^p(\R^k)\setminus\{0\}$.
    
    In particular, if $(d^{k-star}_{x_1,x_2,\dots ,x_k})_{*}(\mu_F)\in L^p(\R^k)$, for some $1\leq p\leq \infty$ and $F'$ is a closed subset of $F$ with positive $\mu_F$ measure, then $\Delta^{k-star}_{x_1,\dots,x_k}(F')$ has positive measure.
\end{lem}

\begin{proof}
    For any Borel set $A\subset \R^k$, 
 \begin{align*} (\Phi^k)_{*}(\mu_{F'})(A)=&\mu_{F'} \left((\Phi^k)^{-1}(A)\right)=\mu_{F} \left(F'\cap(\Phi^k)^{-1}(A)\right)\\
        \leq &\mu_{F} \left((\Phi^k)^{-1}(A)\right)=(\Phi^k)_{*}(\mu_{F})(A).
    \end{align*}

    Therefore $\nu':=(\Phi^k)_{*}(\mu_{F'})\leq \nu=(\Phi^k)_{*}(\mu_{F})$. Then $\nu'$ is absolutely continuous with respect to $\nu$ and by the Radon-Nikodym theorem one has $0\leq \frac{d\nu'}{d\nu}\leq 1$, and consequently, $0\leq \frac{d\nu'}{dx}\in L^p$. 
  
Next, observe that if $\frac{d\nu'}{dx}\in L^p$ then its support must have positive measure. If $p=1$,

$$0<\mu_{F'}(F)=\int \frac{d\nu'}{dx}dx,$$
implies that directly. For $1<p\leq \infty$, it follows from H\"older's inequality since 
$$0<\mu_{F'}(F)=\int_{\R^k} \frac{d\nu'}{dx}dx\leq \left\|\frac{d\mu'}{dx}\right\|_{L^p(\R^k)}\mathcal{L}^k(\text{supp}(d\nu'/dx))^{1/p'}.$$

    The last claim follows from the fact that if $F'$ is closed then $(d^{k-star}_{x_1,x_2,\dots ,x_k})_{*}(\mu_{F'})$ has support contained in $d^{k-star}_{x_1,x_2,\dots ,x_k}(\text{supp} \mu_{F'})\subseteq d^{k-star}_{x_1,x_2,\dots ,x_k}(F')=\Delta^{k-star}_{x_1,\dots ,x_k}(F')$ (\cite[Theorem 1.18]{Mattilabookgeometry}).
\end{proof}

\begin{proof}[Proof of Theorem \ref{thm: firststructuralthm}]

    Fix an $k$-admissible ordering $\mathcal{O}=(v_1,v_2,\dots ,v_l)$ for the pinned graph $(G,\mathcal{P})$ where $\mathcal{P}=(v_1,v_2,\dots, v_m)$ are the pinned vertices. Notice that we can assume without loss of generality that $m\geq 1$ (at least one pinned vertex) because if $m=0$ (no pins), we can add a pin at $v_1$ and still have that $(G,\{v_1\})$ is $k$-admissible with $k$-admissible ordering $\mathcal{O}$.
    
    We are given a compact set $E$ with $\dim(E)>\frac{d+k}{2}$. Let $\mu$ be an $s$-Frostman measure on $E$ with $s>\frac{d+k}{2}$, and extract $l$ separated compact subsets of $E$, say $E_1, E_2, \dots ,E_l$.  From Theorem \ref{thm:inductionstepinstructuralthm1} we know that $\mu_{E_1}\times \mu_{E_2}\times\dots \times  \mu_{E_m}$ almost every $m$-tuple $(x_1,x_2,\dots ,x_m)\in E_1\times E_2\times \dots E_m$ is $(G,\mathcal{P},\mathcal{O},\{E_i\}_{i=1}^l)$ good. 

    We just need to check that if $x:=(x_1,x_2,\dots ,x_m)\in E_1\times E_2\times \dots E_m$ is $(G,\mathcal{P},\mathcal{O},\{E_i\}_{i=1}^l)$ good, then $\Delta^{G}_{\mathcal{P},x_1,\dots, x_k}(E_{m+1},\dots ,E_l)$  has positive measure. We have a sequence of sets $$(E_r)^G_{x,y_{(r-1)}}\subset E_r,\qquad{m+1\leq r\leq l-1}$$ provided from the definition of $(x_1,\dots ,x_m)$ being $(G,\mathcal{P},\mathcal{O},\{E_i\}_{i=1}^l)$ good. 

    \begin{equation*}
    \begin{split}
    &\mathcal{L}^{|\mathcal{E}|}(\Delta^G_{\mathcal{P},x}(E_{m+1},\dots ,E_l))\\
    \geq&\int_{\Delta_{(x_i)_{i\in \eta_{m+1,P}}}^{\epsilon_{m+1}-star}((E_{m+1})^G_{x})} {\int_{\Delta_{{(x_i)_{i\in\eta_{m+2,P}},(y_{j})_{j\in \eta_{m+2,U}}}}^{\epsilon_{m+2}-star}((E_{m+2})^G_{x,y_{(m+1)}})}} \\&{\int_{\Delta_{(x_i)_{i\in\eta_{m+3,P}},(y_{j})_{j\in \eta_{m+3,U}}}^{\epsilon_{m+3}-star}((E_{m+3})^G_{x,y_{(m+2)})})}} \cdots{\int_{\Delta_{(x_i)_{i\in \eta_{l,P}},(y_j)_{j\in \eta_{l,U}}}^{\epsilon_l-star}(E_l)}dT^l}
   \cdots dT^{m+3}T^{m+2}dT^{m+1},
    \end{split}
\end{equation*}
where $dT^{r}=dt^r_1dt^r_2\dots dt^r_{\epsilon_r}$, and for $m\leq r\leq l-1$ we consecutively choose $y_r\in (E_r)^G_{x,y_{(r-1)}}$ so that $d^{\epsilon_{r}-star}_{(x_i)_{i\in \eta_{r,P}},(y_j)_{j\in \eta_{r,U}}}(y_r)=(t^r_1,t^r_2\dots ,t^r_{\epsilon_r})$. Lemma \ref{Lemma: L2lemma} guarantees that in each step of the iterated integral, we are integrating over distance sets of positive measure. That forces the integral in the right-hand side to be positive.
\end{proof}

\begin{proof}[Proof of Theorem \ref{thm:inductionstepinstructuralthm1}]
    We will prove this theorem by induction on the number $n$ of unpinned vertices of the graph. 
    
    \textbf{Base case:} $n=1$. If the graph is connected, it would be exactly an $m$-star with an unpinned vertex $v_{m+1}$ connected to all $v_i$ with $1\leq i\leq m$. In that case, Theorem \ref{thm:k starinIPPS} gives us exactly the conclusion of this theorem. In the disconnected case, there would be at most $m'<m$ pinned vertices $(v_{i_1},v_{i_2},\dots ,v_{i_{m'}})$ connected to $v_{m+1}$ and again one can just use Theorem \ref{thm:k starinIPPS} for $m'$-stars to get that $\Pi_{j=1}^{m'}\mu_{E_{i_j}}$ almost everywhere $(x_{i_1},\dots ,x_{i_{m'}})$ (and no restrictions on the pins $x_i$ with $v_i$ that are not connected to $v_{m+1}$, so still a full measure set with respect to product measure).

    \textbf{Induction step}: Assume that $n\geq 2 $ and that the theorem is true for graphs with up to $n-1$ unpinned vertices. Denote by $(G',\mathcal{P})$ the pinned subgraph obtained by deleting $v_{l}$ (the last added vertex in the $k$-admissible ordering $\mathcal{O}$) and all adjacent edges from $G$; also, let $\mathcal{O}'=(v_1,v_2,\dots,v_{l-1})$ be the $k$-admissible ordering induced in $(G',\mathcal{P})$. Since $G'$ has $n-1$ unpinned vertices, we have that $\mu_{E_1}\times \mu_{E_2}\dots \times \mu_{E_{m}}$ almost every $(x_1,\dots ,x_m)\in E_1\times E_2\times \dots\times E_m$ is $(G',\mathcal{P},\mathcal{O}',\{E_i\}_{i=1}^{l-1})$ good. In other words, there is a subset $\mathcal{E}_m^{G'}\subset E_1\times \dots \times E_m$ such that $(\Pi_{i=1}^m\mu_i)((\mathcal{E}_m^{G'})^c)=0$ and every $x=(x_1,\dots ,x_m)\in \mathcal{E}_m^{G'}$ is good for $(G',\mathcal{P},\mathcal{O}',\{E_i\}_{i=1}^{l-1})$.

    From the definition of $x$ being good we have sets $(E_{m+1})^{G'}_x\subset E_{m+1}$ and for each $y_{m+1}\in (E_{m+1})^{G'}_x$ we have $(E_{m+2})^{G'}_{x,y_{m+1}}\subset E_{m+2}$, and more generally for any $m+1\leq r<l$ there are sets $(E_r)^{G'}_{x,y_{(r-1)}}\subset E_r$ defined as long as $y_{i}\in (E_i)^{G'}_{x,y(i-1)}$ for all $m+1\leq i<r$. In order to update them for $G$, we look at the vertices that are connected to $v_l$. There are $\epsilon_l$ such vertices in total given by indices in $\eta_l=\eta_{l,P}\bigsqcup \eta_{l,U}$. Moreover, from Theorem \ref{thm:k starinIPPS} applied to $\{E_i\colon i\in \eta_{l}\}$ and $E_l$ one has that for $\Pi_{i\in \eta_{l,P}}\mu_{E_i}\times \Pi_{j\in \eta_{j,U}}\mu_{E_j}$ almost every $((x_i)_{i\in \eta_{l,P} }, (y_j)_{j\in \eta_{l,U}})\in\Pi_{i\in \eta_{l,P}}E_i\times \Pi_{j\in \eta_{j,U}}E_j$, one has 
    $$(d^{\epsilon_l-star}_{(x_i)_{i\in \eta_{l,P} }, (y_j)_{j\in \eta_{l,U}}})_{*}(\mu_{E_{l}})\in L^2.$$

As a consequence, there is a set $\mathcal{E}_{\eta_{l,P}}\subset \Pi_{j\in \eta_{l.P}E_j}$ of full $ (\Pi_{j\in \eta_{l,P}}\mu_{E_j})$ measure such that for any $(x_j)_{j\in \eta_{l,P}}\in\mathcal{E}_{\eta_{l,P}}$ one has
$$(d^{\epsilon_l-star}_{(x_i)_{i\in \eta_{l,P} }, (y_j)_{j\in \eta_{l,U}}})_{*}(\mu_{E_{l}})\in L^2$$ 
for $(\Pi_{j\in \eta_{l,U}}\mu_j)$ almost every $(y_j)_{j\in \eta_{l,U}}\in \Pi_{j\in \eta_{l,U}}E_j$.

List the indices in $\eta_{l,U}$ in increasing order $i_1<i_2<\dots <i_{\epsilon_{l,U}}$. For each $x(l,P):=(x_j)_{j\in \eta_{l,P}}\in\mathcal{E}_{\eta_{l,P}}$ there exists a $\mu_{E_{i_1}}$-full measure set $(S_{i_1})_{x(l,P)}\subset E_{i_1}$ such that for all $y_{i_1}\in (S_{i_1})_{x(l,P)}$ there exists a $\mu_{E_{i_2}}$-full measure subset $(S_{i_2})_{x(l,P),y_{i_1}}\subset E_{i_2}$, and more generally for any $2\leq r< \epsilon_{l,U}$ there are sets $(S_{i_r})_{x(l,P),y_{i_1},\dots ,y_{i_{r-1}}}\subset E_{i_r}$ defined as long as $y_{i_j}\in (S_{i_j})_{x(l,P),y_{i_1},\dots, y_{i_{j-1}}}$ for all $m+1\leq j<r$.

%The fact that $\Pi_{i=1}^{m}\mu_{E_i}$ a.e. $(x_1,x_2,\dots ,x_m)$ is $(G',\mathcal{P},\mathcal{O}')$ good can be rewritten as 
%$$(\Pi_{i\in \eta_{l,P}}\mu_{E_i})\times (\Pi_{1\leq j\leq m\colon j\notin \eta_{l,P}}\mu_{E_j})$$ a.e. $((x_i)_{i\in \eta_{l,P}},(x_j)_{1\leq j\leq m\colon j\notin \eta_{l,P}})$ is $(G',\mathcal{P},\mathcal{O}')$ good.
    Let $$\mathcal{E}^G_m:=\{(x_1,\dots ,x_m)\in \mathcal{E}^{G'}_m\colon (x_i)_{i\in \eta_{l,P}}\in \mathcal{E}_{\eta_{l,P}}\} \subset E_1\times E_2\times \dots \times E_m.$$

    It is not hard to see that $(\Pi_{i=1}^m \mu_i) ((\mathcal{E}^G_m)^c)=0$. We claim that every $(x_1,x_2,\dots ,x_m)\in \mathcal{E}^G_m $ is $(G,\mathcal{P},\mathcal{O},\{E_i\}_{i=1}^l)$ good. For every $i_r\in \eta_{l,U}$ one needs to redefine $$(E_{i_r})^G_{x,y(i_r-1)}=(E_{i_r})^{G'}_{x,y(i_r-1)}\cap (S_{i_r})_{x(l,P),y_{i_1},\dots ,y_{i_{r-1}}}\subset E_{i_r},$$
    which still has full $\mu_{E_{i_r}}$-measure. If $(l-1)\in \eta_{l,U}$, we use $(E_{l-1})^{G'}_{x,y(l-2)}:=E_{l-1}$.

    For $m+1\leq j\leq l-1$ such that $j\notin \eta_{l,U}$, we essentially keep it as it was for $G'$, namely
    $$(E_{j})^G_{x,y(j-1)}=(E_{j})^{G'}_{x,y(j-1)}\subset E_{j}, $$
    defined whenever $y_i\in (E_i)^G_{x,y(i-1)}$, for all $1\leq i
<j$.
\end{proof}

The proof of Theorem \ref{thm: secondstructuralthm} follows the same lines as Theorem \ref{thm: firststructuralthm}. We discuss below the main modifications needed.

\begin{proof}[Sketch of the proof of Theorem \ref{thm: secondstructuralthm}]
   As in the proof of Theorem \ref{thm: firststructuralthm}, we can assume without loss of generality that $m\geq 1$. Let $\Phi=(\Phi_{(v_i,v_j)})_{(v_i,v_j)\in \mathcal{E}}$ be the continuous edge vector function of interest. Fix $\mathcal{O}=(v_1,v_2,\dots ,v_m,v_{m+1},\dots ,v_l)$, a construction order for $(G,\mathcal{P})$ that is $\alpha$-suitable for the stars in $(G,\mathcal{P})$ in the sense of Definition \ref{def: suitableforstars}. For each $m+1\leq r\leq l$, let $\epsilon_r$ be the back-degree of $v_r$ according to the ordering $\mathcal{O}$. Let $\eta_r,\eta_{r,P},\eta_{r,U}$ be as before. Define
    
    $$\Phi^r=(\Phi_{(v_i,v_r)})_{i\in \eta_{r}}, m+1\leq r\leq l.$$
In words, $\Phi^r$ is a function with $\epsilon_r$ components of $\Phi$ corresponding to edges between $v_r$ and any other vertex that shares an edge with $v_r$ and which came before in the ordering $\mathcal{O}$.

Define $$d^{\epsilon_r-star,\Phi^r}_{(z_i)_{i\in \eta_r}}(x_r):=(\Phi_{(v_i,v_r)}(x_i,x_r))_{i\in \eta_r}.$$

    We can say that  an $m$-tuple of pins $(x_1,\ldots,x_m)\in E_1\times E_2\times \dots\times  E_m$ is $(G,\mathcal{P},\mathcal{O},\{E_i\}_{i=1}^l, \Phi)$-good if it satisfies the conditions Definition \ref{def: goodpin} with $\left(d_{(x_i)_{i\in\eta_{r,P}},(y_j)_{j\in \eta_{r, U}}}^{\epsilon_{r}-star}\right)_*(\mu_{E_{r}})$ replaced with $$\left(d_{(x_i)_{i\in\eta_{r,P}},(y_j)_{j\in \eta_{r, U}}}^{\epsilon_{r}-star,\Phi^r}\right)_*(\mu_{E_{r}}), m+1\leq r\leq l,$$
and $L^2$ replaced by $L^p$. 
    With proof analogous to the proof of Theorem \ref{thm:inductionstepinstructuralthm1} (by induction of the number of unpinned vertices) one can prove that $(\Pi_{i=1}^{m}\mu_{E_i})$ almost everywhere $m$-tuple $(x_1,\dots,x_m)\in \Pi_{i=1}^{m} E_i$ is $(G,\mathcal{P},\mathcal{O},\{E_i\}_{i=1}^l, \Phi)$-good. For that, it is interesting to observe that by deleting the last unpinned vertex of an $\alpha$-suitable ordering for $G$, we have $\mathcal{O'}=(v_1,\dots ,v_m, v_{m+1}, \dots,v_{l-1})$ is an $\alpha$ suitable construction order for the pinned graph obtained by deleting from $G$ the vertex $v_l$ and all its adjacent vertices.

    As in Theorem \ref{thm: firststructuralthm}, by using Fubini combined with Lemma \ref{Lemma: L2lemma}, every $(G,\mathcal{P},\mathcal{O},\{E_i\}_{i=1}^l, \Phi)$-good $ m$-tuple of pins $(x_1,\dots ,x_m)$ will produce positive measure for the pinned distance set for $\Delta^{G,\Phi}_{\mathcal{P},x_1,\cdots, x_m}(E)$.
\end{proof}

\vskip1.5em

\printbibliography

@misc{pinnedtrees,
      title={Nonempty interior of pinned distance and tree sets}, 
      author={Tainara Borges and Benjamin Foster and Yumeng Ou and Eyvindur Palsson},
      year={2025},
      eprint={2503.15709},
      archivePrefix={arXiv},
      primaryClass={math.CA},
      url={https://arxiv.org/abs/2503.15709}, 
}

@article {IPPS22,
    AUTHOR = {Iosevich, Alex and Pham, Minh-Quy and Pham, Thang and Shen,
              Chun-Yen},
     TITLE = {Pinned simplices and connections to product of sets on
              paraboloids},
   JOURNAL = {Indiana Univ. Math. J.},
  FJOURNAL = {Indiana University Mathematics Journal},
    VOLUME = {74},
      YEAR = {2025},
    NUMBER = {3},
     PAGES = {647--668},
      ISSN = {0022-2518,1943-5258},
   MRCLASS = {52C10 (28A80 42B10)},
  MRNUMBER = {4946877},
}

@InProceedings{CIMP2021,
author="Chatzikonstantinou, Nikolaos
and Iosevich, Alex
and Mkrtchyan, Sevak
and Pakianathan, Jonathan",
editor="Nathanson, Melvyn B.",
title="Rigidity, Graphs and Hausdorff Dimension",
booktitle="Combinatorial and Additive Number Theory IV",
year="2021",
publisher="Springer International Publishing",
address="Cham",
pages="73--106",
isbn="978-3-030-67996-5"
}

@article {PS00,
    AUTHOR = {Peres, Yuval and Schlag, Wilhelm},
     TITLE = {Smoothness of projections, {B}ernoulli convolutions, and the
              dimension of exceptions},
   JOURNAL = {Duke Math. J.},
  FJOURNAL = {Duke Mathematical Journal},
    VOLUME = {102},
      YEAR = {2000},
    NUMBER = {2},
     PAGES = {193--251},
      ISSN = {0012-7094,1547-7398},
   MRCLASS = {42B25 (28A78)},
  MRNUMBER = {1749437},
MRREVIEWER = {Esa\ J\"arvenp\"a\"a},
       DOI = {10.1215/S0012-7094-00-10222-0},
       URL = {https://doi.org/10.1215/S0012-7094-00-10222-0},
}

@article {Falconer85,
    AUTHOR = {Falconer, K. J.},
     TITLE = {On the {H}ausdorff dimensions of distance sets},
   JOURNAL = {Mathematika},
  FJOURNAL = {Mathematika. A Journal of Pure and Applied Mathematics},
    VOLUME = {32},
      YEAR = {1985},
    NUMBER = {2},
     PAGES = {206--212},
      ISSN = {0025-5793},
   MRCLASS = {28A75 (28A05)},
  MRNUMBER = {834490},
MRREVIEWER = {S.\ J.\ Taylor},
       DOI = {10.1112/S0025579300010998},
       URL = {https://doi.org/10.1112/S0025579300010998},
}

@misc{DORZ23,
      title={Weighted refined decoupling estimates and application to {F}alconer distance set problem}, 
      author={Du, Xiumin and Ou, Yumeng and Ren, Kevin and Zhang, Ruixiang},
      year={2023},
      eprint={2309.04501},
      archivePrefix={arXiv},
      primaryClass={math.CA},
      url={https://arxiv.org/abs/2309.04501}, 
}

@article {GIOW20,
    AUTHOR = {Guth, Larry and Iosevich, Alex and Ou, Yumeng and Wang, Hong},
     TITLE = {On {F}alconer's distance set problem in the plane},
   JOURNAL = {Invent. Math.},
  FJOURNAL = {Inventiones Mathematicae},
    VOLUME = {219},
      YEAR = {2020},
    NUMBER = {3},
     PAGES = {779--830},
      ISSN = {0020-9910,1432-1297},
   MRCLASS = {42B20 (28A80)},
  MRNUMBER = {4055179},
MRREVIEWER = {Jonathan\ MacDonald\ Fraser},
       DOI = {10.1007/s00222-019-00917-x},
       URL = {https://doi.org/10.1007/s00222-019-00917-x},
}

@article {Liu19,
    AUTHOR = {Liu, Bochen},
     TITLE = {An {$L^2$}-identity and pinned distance problem},
   JOURNAL = {Geom. Funct. Anal.},
  FJOURNAL = {Geometric and Functional Analysis},
    VOLUME = {29},
      YEAR = {2019},
    NUMBER = {1},
     PAGES = {283--294},
      ISSN = {1016-443X,1420-8970},
   MRCLASS = {28A75 (28A78 42B10)},
  MRNUMBER = {3925111},
MRREVIEWER = {Lars\ Olsen},
       DOI = {10.1007/s00039-019-00482-8},
       URL = {https://doi.org/10.1007/s00039-019-00482-8},
}

@article {MattilaSjolin,
    AUTHOR = {Mattila, Pertti and Sj\"olin, Per},
     TITLE = {Regularity of distance measures and sets},
   JOURNAL = {Math. Nachr.},
  FJOURNAL = {Mathematische Nachrichten},
    VOLUME = {204},
      YEAR = {1999},
     PAGES = {157--162},
      ISSN = {0025-584X,1522-2616},
   MRCLASS = {42B10 (28A75)},
  MRNUMBER = {1705134},
MRREVIEWER = {Joan\ Orobitg},
       DOI = {10.1002/mana.3212040110},
       URL = {https://doi.org/10.1002/mana.3212040110},
}

@book {Mattilabook2015,
    AUTHOR = {Mattila, Pertti},
     TITLE = {Fourier analysis and {H}ausdorff dimension},
    SERIES = {Cambridge Studies in Advanced Mathematics},
    VOLUME = {150},
 PUBLISHER = {Cambridge University Press, Cambridge},
      YEAR = {2015},
     PAGES = {xiv+440},
      ISBN = {978-1-107-10735-9},
   MRCLASS = {28-02 (28A15 28A78 28A80 42B10 60J65)},
  MRNUMBER = {3617376},
MRREVIEWER = {Benjamin\ Steinhurst},
       DOI = {10.1017/CBO9781316227619},
       URL = {https://doi.org/10.1017/CBO9781316227619},
}

@article {IMT12,
    AUTHOR = {Iosevich, Alex and Mourgoglou, Mihalis and Taylor, Krystal},
     TITLE = {On the {M}attila-{S}j\"olin theorem for distance sets},
   JOURNAL = {Ann. Acad. Sci. Fenn. Math.},
  FJOURNAL = {Annales Academi\ae\ Scientiarum Fennic\ae. Mathematica},
    VOLUME = {37},
      YEAR = {2012},
    NUMBER = {2},
     PAGES = {557--562},
      ISSN = {1239-629X,1798-2383},
   MRCLASS = {28A75 (42B20 52C10)},
  MRNUMBER = {2987085},
MRREVIEWER = {Alain\ Rivi\`ere},
       DOI = {10.5186/aasfm.2012.3732},
       URL = {https://doi.org/10.5186/aasfm.2012.3732},
}

@article {GILP15,
    AUTHOR = {Greenleaf, Allan and Iosevich, Alex and Liu, Bochen and
              Palsson, Eyvindur},
     TITLE = {A group-theoretic viewpoint on {E}rd\"os-{F}alconer problems
              and the {M}attila integral},
   JOURNAL = {Rev. Mat. Iberoam.},
  FJOURNAL = {Revista Matem\'atica Iberoamericana},
    VOLUME = {31},
      YEAR = {2015},
    NUMBER = {3},
     PAGES = {799--810},
      ISSN = {0213-2230,2235-0616},
   MRCLASS = {42B20 (52C10)},
  MRNUMBER = {3420476},
MRREVIEWER = {Tuomas\ P.\ Hyt\"onen},
       DOI = {10.4171/RMI/854},
       URL = {https://doi.org/10.4171/RMI/854},
}

@book {Mattilabookgeometry,
    AUTHOR = {Mattila, Pertti},
     TITLE = {Geometry of sets and measures in {E}uclidean spaces},
    SERIES = {Cambridge Studies in Advanced Mathematics},
    VOLUME = {44},
      NOTE = {Fractals and rectifiability},
 PUBLISHER = {Cambridge University Press, Cambridge},
      YEAR = {1995},
     PAGES = {xii+343},
      ISBN = {0-521-46576-1; 0-521-65595-1},
   MRCLASS = {28A75 (49Q20)},
  MRNUMBER = {1333890},
MRREVIEWER = {Harold\ Parks},
       DOI = {10.1017/CBO9780511623813},
       URL = {https://doi.org/10.1017/CBO9780511623813},
}

@article {EIT11,
    AUTHOR = {Eswarathasan, Suresh and Iosevich, Alex and Taylor, Krystal},
     TITLE = {Fourier integral operators, fractal sets, and the regular
              value theorem},
   JOURNAL = {Adv. Math.},
  FJOURNAL = {Advances in Mathematics},
    VOLUME = {228},
      YEAR = {2011},
    NUMBER = {4},
     PAGES = {2385--2402},
      ISSN = {0001-8708,1090-2082},
   MRCLASS = {42B08 (28A80 44A12)},
  MRNUMBER = {2836125},
MRREVIEWER = {Herv\'e\ Queff\'elec},
       DOI = {10.1016/j.aim.2011.07.012},
       URL = {https://doi.org/10.1016/j.aim.2011.07.012},
}

@misc{BMS24,
title={Pinned Dot Product Set Estimates}, 
      author={Paige Bright and Caleb Marshall and Steven Senger},
      year={2024},
      eprint={2412.17985},
      archivePrefix={arXiv},
      primaryClass={math.CA},
      url={https://arxiv.org/abs/2412.17985},
}

@article {GIT21,
    AUTHOR = {Greenleaf, Allan and Iosevich, Alex and Taylor, Krystal},
     TITLE = {Configuration sets with nonempty interior},
   JOURNAL = {J. Geom. Anal.},
  FJOURNAL = {Journal of Geometric Analysis},
    VOLUME = {31},
      YEAR = {2021},
    NUMBER = {7},
     PAGES = {6662--6680},
      ISSN = {1050-6926,1559-002X},
   MRCLASS = {28A80 (35S30 44A12)},
  MRNUMBER = {4289240},
       DOI = {10.1007/s12220-019-00288-y},
       URL = {https://doi.org/10.1007/s12220-019-00288-y},
}

@article {GI12,
    AUTHOR = {Greenleaf, Allan and Iosevich, Alex},
     TITLE = {On triangles determined by subsets of the {E}uclidean plane,
              the associated bilinear operators and applications to discrete
              geometry},
   JOURNAL = {Anal. PDE},
  FJOURNAL = {Analysis \& PDE},
    VOLUME = {5},
      YEAR = {2012},
    NUMBER = {2},
     PAGES = {397--409},
      ISSN = {2157-5045,1948-206X},
   MRCLASS = {42B15 (52C10)},
  MRNUMBER = {2970712},
MRREVIEWER = {Andreas\ Seeger},
       DOI = {10.2140/apde.2012.5.397},
       URL = {https://doi.org/10.2140/apde.2012.5.397},
}

@book {BMP05,
    AUTHOR = {Brass, Peter and Moser, William and Pach, J\'anos},
     TITLE = {Research problems in discrete geometry},
 PUBLISHER = {Springer, New York},
      YEAR = {2005},
     PAGES = {xii+499},
      ISBN = {978-0387-23815-8; 0-387-23815-8},
   MRCLASS = {52-02 (05-02)},
  MRNUMBER = {2163782},
MRREVIEWER = {W.\ Kuperberg},
}

@incollection {EHI13,
    AUTHOR = {Erdogan, Burak and Hart, Derrick and Iosevich, Alex},
     TITLE = {Multiparameter projection theorems with applications to
              sums-products and finite point configurations in the
              {E}uclidean setting},
 BOOKTITLE = {Recent advances in harmonic analysis and applications},
    SERIES = {Springer Proc. Math. Stat.},
    VOLUME = {25},
     PAGES = {93--103},
 PUBLISHER = {Springer, New York},
      YEAR = {2013},
      ISBN = {978-1-4614-4565-4; 978-1-4614-4564-7},
   MRCLASS = {28A80 (42B08)},
  MRNUMBER = {3066881},
MRREVIEWER = {Li-Feng\ Xi},
       DOI = {10.1007/978-1-4614-4565-4\_11},
       URL = {https://doi.org/10.1007/978-1-4614-4565-4_11},
}

@article {PRA23,
    AUTHOR = {Palsson, Eyvindur Ari and Romero Acosta, Francisco},
     TITLE = {A {M}attila-{S}j\"olin theorem for triangles},
   JOURNAL = {J. Funct. Anal.},
  FJOURNAL = {Journal of Functional Analysis},
    VOLUME = {284},
      YEAR = {2023},
    NUMBER = {6},
     PAGES = {Paper No. 109814, 20},
      ISSN = {0022-1236,1096-0783},
   MRCLASS = {28A78 (42B20 52C10)},
  MRNUMBER = {4530888},
MRREVIEWER = {Vladimir\ Eiderman},
       DOI = {10.1016/j.jfa.2022.109814},
       URL = {https://doi.org/10.1016/j.jfa.2022.109814},
}

@article {PRA25,
    AUTHOR = {Palsson, Eyvindur Ari and Romero Acosta, Francisco},
     TITLE = {A {M}attila-{S}j\"olin theorem for simplices in low
              dimensions},
   JOURNAL = {Math. Ann.},
  FJOURNAL = {Mathematische Annalen},
    VOLUME = {391},
      YEAR = {2025},
    NUMBER = {1},
     PAGES = {1123--1146},
      ISSN = {0025-5831,1432-1807},
   MRCLASS = {28A75 (28A78 42B20 52A20)},
  MRNUMBER = {4846807},
MRREVIEWER = {Bochen\ Liu},
       DOI = {10.1007/s00208-024-02948-z},
       URL = {https://doi.org/10.1007/s00208-024-02948-z},
}

@article {GGIP15,
    AUTHOR = {Grafakos, Loukas and Greenleaf, Allan and Iosevich, Alex and
              Palsson, Eyvindur},
     TITLE = {Multilinear generalized {R}adon transforms and point
              configurations},
   JOURNAL = {Forum Math.},
  FJOURNAL = {Forum Mathematicum},
    VOLUME = {27},
      YEAR = {2015},
    NUMBER = {4},
     PAGES = {2323--2360},
      ISSN = {0933-7741,1435-5337},
   MRCLASS = {42B15 (05D05)},
  MRNUMBER = {3365800},
       DOI = {10.1515/forum-2013-0128},
       URL = {https://doi.org/10.1515/forum-2013-0128},
}

@article {GIT22,
    AUTHOR = {Greenleaf, Allan and Iosevich, Alex and Taylor, Krystal},
     TITLE = {On {$k$}-point configuration sets with nonempty interior},
   JOURNAL = {Mathematika},
  FJOURNAL = {Mathematika. A Journal of Pure and Applied Mathematics},
    VOLUME = {68},
      YEAR = {2022},
    NUMBER = {1},
     PAGES = {163--190},
      ISSN = {0025-5793,2041-7942},
   MRCLASS = {28A75 (28A80 52C10 58J40)},
  MRNUMBER = {4405974},
MRREVIEWER = {Xiumin\ Du},
       DOI = {10.1112/mtk.12114},
       URL = {https://doi.org/10.1112/mtk.12114},
}

@article {GIT24,
    AUTHOR = {Greenleaf, Allan and Iosevich, Alex and Taylor, Krystal},
     TITLE = {Nonempty interior of configuration sets via microlocal
              partition optimization},
   JOURNAL = {Math. Z.},
  FJOURNAL = {Mathematische Zeitschrift},
    VOLUME = {306},
      YEAR = {2024},
    NUMBER = {4},
     PAGES = {Paper No. 66, 20},
      ISSN = {0025-5874,1432-1823},
   MRCLASS = {28A75 (28A80 52C10 58J40)},
  MRNUMBER = {4716767},
MRREVIEWER = {Stefan\ Steinerberger},
       DOI = {10.1007/s00209-024-03466-z},
       URL = {https://doi.org/10.1007/s00209-024-03466-z},
}

@article {GIT25,
    AUTHOR = {Greenleaf, Allan and Iosevich, Alex and Taylor, Krystal},
     TITLE = {Realizing trees of configurations in thin sets},
   JOURNAL = {Pacific J. Math.},
  FJOURNAL = {Pacific Journal of Mathematics},
    VOLUME = {335},
      YEAR = {2025},
    NUMBER = {2},
     PAGES = {355--372},
      ISSN = {0030-8730,1945-5844},
   MRCLASS = {28A75 (42B35)},
  MRNUMBER = {4904870},
       DOI = {10.2140/pjm.2025.335.355},
       URL = {https://doi.org/10.2140/pjm.2025.335.355},
}

@article {GIP17,
    AUTHOR = {Greenleaf, Allan and Iosevich, Alex and Pramanik, Malabika},
     TITLE = {On necklaces inside thin subsets of {$\Bbb R^d$}},
   JOURNAL = {Math. Res. Lett.},
  FJOURNAL = {Mathematical Research Letters},
    VOLUME = {24},
      YEAR = {2017},
    NUMBER = {2},
     PAGES = {347--362},
      ISSN = {1073-2780,1945-001X},
   MRCLASS = {28A80 (43A46)},
  MRNUMBER = {3685274},
       DOI = {10.4310/MRL.2017.v24.n2.a4},
       URL = {https://doi.org/10.4310/MRL.2017.v24.n2.a4},
}

@article {Matula,
    AUTHOR = {Matula, David W. and Beck, Leland L.},
     TITLE = {Smallest-last ordering and clustering and graph coloring
              algorithms},
   JOURNAL = {J. Assoc. Comput. Mach.},
  FJOURNAL = {Journal of the Association for Computing Machinery},
    VOLUME = {30},
      YEAR = {1983},
    NUMBER = {3},
     PAGES = {417--427},
      ISSN = {0004-5411,1557-735X},
   MRCLASS = {68E10 (05C15 68C25)},
  MRNUMBER = {709826},
MRREVIEWER = {Charles\ J.\ Colbourn},
       DOI = {10.1145/2402.322385},
       URL = {https://doi.org/10.1145/2402.322385},
}

@article {McDonald21,
    AUTHOR = {McDonald, Alex},
     TITLE = {Areas spanned by point configurations in the plane},
   JOURNAL = {Proc. Amer. Math. Soc.},
  FJOURNAL = {Proceedings of the American Mathematical Society},
    VOLUME = {149},
      YEAR = {2021},
    NUMBER = {5},
     PAGES = {2035--2049},
      ISSN = {0002-9939,1088-6826},
   MRCLASS = {28A75},
  MRNUMBER = {4232196},
       DOI = {10.1090/proc/15348},
       URL = {https://doi.org/10.1090/proc/15348},
}

@article {GM22,
    AUTHOR = {Galo, Belmiro and McDonald, Alex},
     TITLE = {Volumes spanned by {$k$}-point configurations in {$\Bbb R^d$}},
   JOURNAL = {J. Geom. Anal.},
  FJOURNAL = {Journal of Geometric Analysis},
    VOLUME = {32},
      YEAR = {2022},
    NUMBER = {1},
     PAGES = {Paper No. 23, 26},
      ISSN = {1050-6926,1559-002X},
   MRCLASS = {28A75},
  MRNUMBER = {4349927},
MRREVIEWER = {Stefan\ Steinerberger},
       DOI = {10.1007/s12220-021-00763-5},
       URL = {https://doi.org/10.1007/s12220-021-00763-5},
}

@article {SY25,
    AUTHOR = {Shmerkin, Pablo and Yavicoli, Alexia},
     TITLE = {On the volumes of simplices determined by a subset of {$\Bbb
              R^d$}},
   JOURNAL = {Ann. Fenn. Math.},
  FJOURNAL = {Annales Fennici Mathematici},
    VOLUME = {50},
      YEAR = {2025},
    NUMBER = {1},
     PAGES = {97--108},
      ISSN = {2737-0690,2737-114X},
   MRCLASS = {28A78 (11B25 28A12 28A80)},
  MRNUMBER = {4877951},
       DOI = {10.54330/afm.159807},
       URL = {https://doi.org/10.54330/afm.159807},
}

@article {Harangi11,
    AUTHOR = {Harangi, Viktor},
     TITLE = {Large dimensional sets not containing a given angle},
   JOURNAL = {Cent. Eur. J. Math.},
  FJOURNAL = {Central European Journal of Mathematics},
    VOLUME = {9},
      YEAR = {2011},
    NUMBER = {4},
     PAGES = {757--764},
      ISSN = {1895-1074,1644-3616},
   MRCLASS = {28A78 (28A80)},
  MRNUMBER = {2805308},
MRREVIEWER = {John\ A.\ Rock},
       DOI = {10.2478/s11533-011-0043-x},
       URL = {https://doi.org/10.2478/s11533-011-0043-x},
}

@article {HKKMMMS13,
    AUTHOR = {Harangi, Viktor and Keleti, Tam\'as and Kiss, Gergely and
              Maga, P\'eter and M\'ath\'e, Andr\'as and Mattila, Pertti and
              Strenner, Bal\'azs},
     TITLE = {How large dimension guarantees a given angle?},
   JOURNAL = {Monatsh. Math.},
  FJOURNAL = {Monatshefte f\"ur Mathematik},
    VOLUME = {171},
      YEAR = {2013},
    NUMBER = {2},
     PAGES = {169--187},
      ISSN = {0026-9255,1436-5081},
   MRCLASS = {28A78},
  MRNUMBER = {3077930},
MRREVIEWER = {Alain\ Rivi\`ere},
       DOI = {10.1007/s00605-012-0455-0},
       URL = {https://doi.org/10.1007/s00605-012-0455-0},
}

@article {IMP16,
    AUTHOR = {Iosevich, Alex and Mourgoglou, Mihalis and Palsson, Eyvindur
              Ari},
     TITLE = {On angles determined by fractal subsets of the {E}uclidean
              space},
   JOURNAL = {Math. Res. Lett.},
  FJOURNAL = {Mathematical Research Letters},
    VOLUME = {23},
      YEAR = {2016},
    NUMBER = {6},
     PAGES = {1737--1759},
      ISSN = {1073-2780,1945-001X},
   MRCLASS = {28A80 (52C10)},
  MRNUMBER = {3621105},
MRREVIEWER = {Maria\ Moszy\'nska},
       DOI = {10.4310/MRL.2016.v23.n6.a8},
       URL = {https://doi.org/10.4310/MRL.2016.v23.n6.a8},
}

@article {Mathe17,
    AUTHOR = {M\'{a}th\'e, Andr\'{a}s},
     TITLE = {Sets of large dimension not containing polynomial
              configurations},
   JOURNAL = {Adv. Math.},
  FJOURNAL = {Advances in Mathematics},
    VOLUME = {316},
      YEAR = {2017},
     PAGES = {691--709},
      ISSN = {0001-8708,1090-2082},
   MRCLASS = {28A78},
  MRNUMBER = {3672917},
MRREVIEWER = {Jonathan\ MacDonald\ Fraser},
       DOI = {10.1016/j.aim.2017.01.002},
       URL = {https://doi.org/10.1016/j.aim.2017.01.002},
}

@article {BIT16,
    AUTHOR = {Bennett, Michael and Iosevich, Alexander and Taylor, Krystal},
     TITLE = {Finite chains inside thin subsets of {$\Bbb{R}^d$}},
   JOURNAL = {Anal. PDE},
  FJOURNAL = {Analysis \& PDE},
    VOLUME = {9},
      YEAR = {2016},
    NUMBER = {3},
     PAGES = {597--614},
      ISSN = {2157-5045,1948-206X},
   MRCLASS = {28A75 (42B10)},
  MRNUMBER = {3518531},
MRREVIEWER = {Gareth\ Speight},
       DOI = {10.2140/apde.2016.9.597},
       URL = {https://doi.org/10.2140/apde.2016.9.597},
}

@incollection {IT19,
    AUTHOR = {Iosevich, A. and Taylor, K.},
     TITLE = {Finite trees inside thin subsets of {$\Bbb R^d$}},
 BOOKTITLE = {Modern methods in operator theory and harmonic analysis},
    SERIES = {Springer Proc. Math. Stat.},
    VOLUME = {291},
     PAGES = {51--56},
 PUBLISHER = {Springer, Cham},
      YEAR = {2019},
      ISBN = {978-3-030-26748-3; 978-3-030-26747-6},
   MRCLASS = {42B10 (05C05 28A80)},
  MRNUMBER = {4008977},
MRREVIEWER = {Peter\ R.\ Massopust},
       DOI = {10.1007/978-3-030-26748-3\_3},
       URL = {https://doi.org/10.1007/978-3-030-26748-3_3},
}

@article {OT22,
    AUTHOR = {Ou, Yumeng and Taylor, Krystal},
     TITLE = {Finite point configurations and the regular value theorem in a
              fractal setting},
   JOURNAL = {Indiana Univ. Math. J.},
  FJOURNAL = {Indiana University Mathematics Journal},
    VOLUME = {71},
      YEAR = {2022},
    NUMBER = {4},
     PAGES = {1707--1761},
      ISSN = {0022-2518,1943-5258},
   MRCLASS = {42B10 (28A78 52C10)},
  MRNUMBER = {4481098},
MRREVIEWER = {Rami\ Ayoush},
}

@misc{IMMM25,
      title={The VC-dimension and point configurations in $\mathbb{R}^d$}, 
      author={Alex Iosevich and Akos Magyar and Alex McDonald and Brian McDonald},
      year={2025},
      eprint={2510.13984},
      archivePrefix={arXiv},
      primaryClass={math.CA},
      url={https://arxiv.org/abs/2510.13984}, 
}

@article {LW70,
    AUTHOR = {Lick, Don R. and White, Arthur T.},
     TITLE = {{$k$}-degenerate graphs},
   JOURNAL = {Canadian J. Math.},
  FJOURNAL = {Canadian Journal of Mathematics. Journal Canadien de
              Math\'ematiques},
    VOLUME = {22},
      YEAR = {1970},
     PAGES = {1082--1096},
      ISSN = {0008-414X,1496-4279},
   MRCLASS = {05.55},
  MRNUMBER = {266812},
MRREVIEWER = {H.\ V.\ Kronk},
       DOI = {10.4153/CJM-1970-125-1},
       URL = {https://doi.org/10.4153/CJM-1970-125-1},
}

@article {IJMD21,
    AUTHOR = {Iosevich, A. and Jardine, G. and McDonald, B.},
     TITLE = {Cycles of arbitrary length in distance graphs on {$\Bbb
              F_q^d$}},
      NOTE = {English version published in Proc. Steklov Inst. Math. {\bf
              314} (2021), no. 1, 27--43.},
   JOURNAL = {Tr. Mat. Inst. Steklova},
  FJOURNAL = {Trudy Matematicheskogo Instituta Imeni V. A. Steklova},
    VOLUME = {314},
      YEAR = {2021},
     PAGES = {31--48},
      ISSN = {0371-9685,3034-1809},
   MRCLASS = {05C12 (11T24)},
  MRNUMBER = {4324083},
       DOI = {10.4213/tm4189},
       URL = {https://doi.org/10.4213/tm4189},
}

@article {IP19,
    AUTHOR = {Iosevich, Alex and Parshall, Hans},
     TITLE = {Embedding distance graphs in finite field vector spaces},
   JOURNAL = {J. Korean Math. Soc.},
  FJOURNAL = {Journal of the Korean Mathematical Society},
    VOLUME = {56},
      YEAR = {2019},
    NUMBER = {6},
     PAGES = {1515--1528},
      ISSN = {0304-9914,2234-3008},
   MRCLASS = {52C10 (05C62 11T23)},
  MRNUMBER = {4015984},
MRREVIEWER = {Iskander\ Aliev},
       DOI = {10.4134/JKMS.j180776},
       URL = {https://doi.org/10.4134/JKMS.j180776},
}

@article {BFHIJPS24,
    AUTHOR = {Bright, Paige and Fang, Xinyu and Heritage, Barrett and
              Iosevich, Alex and Jiang, Tingsong and Parshall, Hans and Sun,
              Maxwell},
     TITLE = {Generalized point configurations in {$\Bbb F_q^d$}},
   JOURNAL = {Finite Fields Appl.},
  FJOURNAL = {Finite Fields and their Applications},
    VOLUME = {99},
      YEAR = {2024},
     PAGES = {Paper No. 102472, 28},
      ISSN = {1071-5797,1090-2465},
   MRCLASS = {52C10 (11T23)},
  MRNUMBER = {4783545},
       DOI = {10.1016/j.ffa.2024.102472},
       URL = {https://doi.org/10.1016/j.ffa.2024.102472},
}

@article {LM20,
    AUTHOR = {Lyall, Neil and Magyar, \'Akos},
     TITLE = {Distance graphs and sets of positive upper density in
              {$\Bbb{R}^d$}},
   JOURNAL = {Anal. PDE},
  FJOURNAL = {Analysis \& PDE},
    VOLUME = {13},
      YEAR = {2020},
    NUMBER = {3},
     PAGES = {685--700},
      ISSN = {2157-5045,1948-206X},
   MRCLASS = {11B30 (05C55 05C62 28A75 52C10)},
  MRNUMBER = {4085119},
MRREVIEWER = {Hans\ Parshall},
       DOI = {10.2140/apde.2020.13.685},
       URL = {https://doi.org/10.2140/apde.2020.13.685},
}

@misc{BFOPRA26,
      title={On volume vectors determined by hypergraphs in thin subsets of Euclidean space}, 
      author={Tainara Borges and Benjamin Foster and Yumeng Ou and Eyvindur Palsson and Francisco {Romero Acosta}},
      year={2026},
      howpublished={preprint},
}
%\bibliographystyle{alpha}
%\bibliography{sources}

%\vskip5em

\end{document}